\documentclass[10pt]{amsart}

\usepackage{float}
\usepackage{amscd,amsmath,amssymb,amsfonts,amsthm,ascmac}
\usepackage{enumerate,mathrsfs,stmaryrd,latexsym, comment, mathtools, mathdots}

\usepackage{amsthm, amssymb}
\usepackage{bm}
\usepackage[leqno]{amsmath}
\usepackage{caption}
\usepackage{subcaption}
\usepackage{hyperref}
\usepackage{relsize}
\usepackage{booktabs}
\usepackage{a4wide}

\usepackage[all]{xy}

\usepackage{tikz-cd}
\usepackage{tikz}
\usetikzlibrary{snakes, 
	3d, matrix,decorations.pathreplacing,calc,decorations.pathmorphing, fit, patterns, trees}
\usetikzlibrary {positioning}
\usepackage{tikz-3dplot}

\numberwithin{equation}{section}
\allowdisplaybreaks[1]

	\allowdisplaybreaks

\newcommand{\defi}[1]{{\textit{#1}}}
\newcommand{\Xvw}[2]{{X^{#1}_{#2}}}
\newcommand{\Q}{{\mathsf{Q}}}

\newcommand{\flag}{{\mathcal{F} \ell}}

\newcommand{\C}{{\mathbb{C}}}

\newcommand{\R}{{\mathbb{R}}}

\newcommand{\Z}{{\mathbb{Z}}}

\newcommand{\polygon}[1]{\mathsf{P}_{#1}}

\newcommand{\B}{\mathcal{B}}

\renewcommand{\ll}{\mathsf{L}}
\newcommand{\rr}{\mathsf{R}}

\newcommand{\T}{\mathcal{T}}
\newcommand{\lT}{\T_T^{\ll}}
\newcommand{\rT}{\T_T^{\rr}}

\newcommand{\head}{\mathsf{h}}
\newcommand{\tail}{\mathsf{t}}

\newcommand{\uh}[1]{{#1}_{\head}}
\newcommand{\ut}[1]{{#1}_{\tail}}

\DeclareMathOperator{\GL}{GL}
\DeclareMathOperator{\Cone}{Cone}
\DeclareMathOperator{\PC}{PC}

\def\Sn{\mathfrak{S}_{n+1}}
\def\Tb{\mathbb{T}}

\newtheorem{theorem}{Theorem}[section]
\newtheorem{lemma}[theorem]{Lemma}
\newtheorem{proposition}[theorem]{Proposition}
\newtheorem{corollary}[theorem]{Corollary}

\newtheorem{Question}[theorem]{Question}

\theoremstyle{definition}
\newtheorem{example}[theorem]{Example}
\newtheorem{definition}[theorem]{Definition}
\newtheorem{remark}[theorem]{Remark}

\hypersetup{colorlinks}
\hypersetup{
	unicode=false,          
	pdftoolbar=true,        
	pdfmenubar=true,        
	pdffitwindow=false,     
	pdfstartview={FitH},    
	pdftitle={My title},    
	pdfauthor={Author},     
	pdfsubject={Subject},   
	pdfcreator={Creator},   
	pdfproducer={Producer}, 
	pdfkeywords={keyword1} {key2} {key3}, 
	pdfnewwindow=true,      
	colorlinks=true,       
	linkcolor=blue,          
	citecolor=red,        
	filecolor=violet,      
	urlcolor=violet       
}

\begin{document}

\author[Eunjeong Lee]{Eunjeong Lee}
\address[E. Lee]{Center for Geometry and Physics, Institute for Basic Science (IBS), Pohang 37673, Republic of Korea}
\curraddr{Department of Mathematics, Chungbuk National University, Cheongju 28644, Republic of Korea}
\email{eunjeong.lee@chungbuk.ac.kr}

\author[Mikiya Masuda]{Mikiya Masuda}
\address[M. Masuda]{Osaka City University Advanced Mathematics Institute (OCAMI) \& Department of Mathematics, Graduate School of Science, Osaka City University, Sumiyoshi-ku, Sugimoto, 558-8585, Osaka, Japan}
\email{mikiyamsd@gmail.com}

\author[Seonjeong Park]{Seonjeong Park}
\address[S. Park]{Department of Mathematics Education, Jeonju University, Jeonju 55069, Republic of Korea}
\email{seonjeongpark@jj.ac.kr}

\thanks{Lee was supported by the Institute for Basic Science (IBS-R003-D1). Masuda was supported in part by JSPS Grant-in-Aid for Scientific Research 19K03472 and a HSE University Basic Research Program. Park was supported by the Basic Science Research Program through the National Research Foundation of Korea (NRF) funded by the Government of Korea (NRF-2018R1A6A3A11047606, NRF-2020R1A2C1A01011045). This work was partly supported by Osaka City University Advanced Mathematical Institute (MEXT Joint Usage/Research Center on Mathematics and Theoretical Physics JPMXP0619217849).}



\title[Toric Richardson varieties of Catalan type]{Toric Richardson varieties of Catalan type and Wedderburn--Etherington numbers}

\date{\today}



\maketitle

\begin{abstract} 
We associate a complete non-singular fan with a polygon triangulation. Such a fan appears from a certain toric Richardson variety, called of Catalan type introduced in this paper. A toric Richardson variety of Catalan type is a Fano Bott manifold. We show that toric Richardson varieties of Catalan type are classified up to isomorphism in terms of unordered binary trees. In particular, the number of isomorphism classes of $n$-dimensional toric Richardson varieties of Catalan type is the $(n+1)$th Wedderburn--Etherington number. 
\end{abstract}

\setcounter{tocdepth}{1} \tableofcontents

\section{Introduction} 

In this paper, we restrict our concern to Lie type $A$.  
Let $G=\GL_{n+1}(\C)$, $B$ the Borel subgroup consisting of upper triangular matrices in $G$, and $\Tb$  the complex torus consisting of diagonal matrices in $G$.
The \emph{flag variety} is a smooth projective variety defined by the homogeneous space~$G/B$. The complex torus $\Tb$ acts on $G/B$ by the left multiplication.
As exhibited by Gelfand and Serganova~\cite{GelfandSerganova} (also, see~\cite{GGMS87}), the study of the action of the torus $\Tb$ on the flag variety $G/B$  provides a fruitful connection between combinatorics on the symmetric group and equivariant algebraic geometry on the flag variety. For instance, through a moment map
\[
\mu\colon G/B\to \R^{n+1},
\]
we see how the closures of $\Tb$-orbits in $G/B$ are related to the combinatorics of polytopes, called \defi{Coxeter matroid polytopes}. 
The set of $\Tb$-fixed points in $G/B$ can naturally be identified with the symmetric group $\Sn$ on $n+1$ letters and we have 
\[\mu(z)=(z^{-1}(1),\dots,z^{-1}(n+1))\quad \text{for $z\in\Sn$},\]
see~\cite[Lemma~3.1]{LMP2021}.

For a pair $(v,w)$ of elements in $\Sn$ satisfying $v \leq w$ in the Bruhat order, Kodama and Williams~\cite{KodamaWilliams} introduced the \defi{Bruhat interval polytope} $\Q_{v,w}$ which is defined as the convex hull of the points $(z(1),\dots,z(n+1))$ in $\R^{n+1}$ for all $z$ with $v\le z\le w$. The combinatorial properties of Bruhat interval polytopes were further investigated by  Tsukerman and Williams in~\cite{TsukermanWilliams}.
The Bruhat interval polytope $\Q_{v,w}$ is an example of a Coxeter matroid polytope and is the image of the Richardson variety $\Xvw{v^{-1}}{w^{-1}}$ by the moment map $\mu$. Here, $\Xvw{v^{-1}}{w^{-1}}$ is the intersection of the Schubert variety~$X_{w^{-1}} {\coloneqq \overline{B w^{-1}B/B} \subset G/B}$ and the opposite Schubert variety $w_0 X_{w_0 v^{-1}}$, where $w_0$ is the longest element in $\Sn$. Note that $\Xvw{v^{-1}}{w^{-1}}=X_{w^{-1}}$ when $v$ is the identity element. 
We remark that the action of $\Tb$ on $G/B$ leaves any Richardson variety.

It is known that 
\[
\dim_{\R} \Q_{v,w}\le \ell (w)-\ell(v)  = \dim_{\C} \Xvw{v^{-1}}{w^{-1}},
\]
where $\ell$ is the length function on $\Sn$. The Richardson variety $\Xvw{v^{-1}}{w^{-1}}$ is a toric variety with respect to the $\Tb$-action if and only if $\dim_{\R} \Q_{v,w}=\ell (w)-\ell(v)$. In this case, the fan of $\Xvw{v^{-1}}{w^{-1}}$ is the normal fan of $\Q_{v,w}$. Every \emph{toric} Schubert variety is smooth, but a \emph{toric} Richardson variety is not necessarily smooth. It is smooth if and only if the corresponding Bruhat interval polytope~$\Q_{v,w}$ is combinatorially equivalent to a cube (see~\cite[Proposition~5.6]{LMP2021}). This means that a smooth toric Richardson variety is a Bott manifold that is the total space of an iterated $\C P^1$-bundle over a point, where each $\C P^1$-bundle is the projectivization of the Whitney sum of two line bundles. Indeed, Hirzebruch surfaces are $2$-dimensional Bott manifolds. 

It is known that when $v$ is the identity element $e$, {the Bruhat interval 
polytope} $\Q_{e,w}$ is combinatorially equivalent to a cube if and only if $w$ is 
a product of distinct simple transpositions $s_i$ interchanging $i$ and $i+1$ 
(see~\cite{Fan98, Karu13Schubert, LM2020}). Such characterization of $w$ is not 
known for general~$v$ but there are many pairs $(v,w)$ such that $\Q_{v,w}$ is 
combinatorially equivalent to a cube (see~\cite{LMP2021}). Among them, the 
following pair of $v$ and $w$ is the simplest form: 
\begin{equation} \label{eq:simple_pair_vw}
w=vs_1s_2\cdots s_n\quad (\text{or}\ w=vs_n\cdots s_2s_1)\quad \text{and}\quad \ell(w)-\ell(v)=n.
\end{equation}

In this paper, we associate a complete non-singular fan of dimension $n$ with a triangulation of a convex $(n+2)$-gon $\polygon{n+2}$ and see that such a fan is the normal fan of $\Q_{v,w}$ for the pair $(v,w)$ in~\eqref{eq:simple_pair_vw} and vice versa. As is well-known, the number of triangulations of $\polygon{n+2}$ is the Catalan number $C_n=\frac{1}{n+1}\binom{2n}{n}$, so we say that a toric Richardson variety (or a toric variety) is \emph{of  Catalan type} if its fan is associated with a polygon triangulation. A toric (Richardson) variety of Catalan type is not only a Bott manifold but also Fano (Lemma~\ref{lemm:Fano}). We note that not all smooth toric Richardson varieties are Fano. Indeed, there is a toric Schubert variety which is not Fano although any toric Schubert variety {(of Lie type $A$)} is weak Fano (see~\cite{LMP_toric_Schubert}). 

There is a well-known bijection between the set of triangulations of $\polygon{n+2}$ and the set of (rooted) binary trees with $n$ vertices. We note that a binary tree is \textit{ordered}, which means that an ordering is specified for the children of each vertex. We show that two polygon triangulations produce isomorphic toric (Richardson) varieties of Catalan type if and only if the corresponding binary trees are isomorphic as rooted trees when we forget the orderings. Namely, we have the following. 

\begin{theorem}[Theorem~\ref{theo:fan_and_toric_of_Catalan} and Corollary~\ref{cor_enumerate}] 
The set of isomorphism classes of $n$-dimensional toric 
\textup{(}Richardson\textup{)} varieties of Catalan type bijectively corresponds to 
the set of \emph{unordered} binary trees with $n$ vertices, where the cardinality 
of the latter set is known as the Wedderburn--Etherington number $b_{n+1}$. 
\end{theorem}

The Wedderburn--Etherington number $b_n$ $(n\ge 1)$ is the number of ways of parenthesizing a string of $n$ letters, subject to a commutative (but nonassociative) binary operation and appears in counting several different objects (see \href{https://oeis.org/A001190}{Sequence A001190} in OEIS~\cite{OEIS}, \cite[A56 in p.133]{Stanley_Catalan}). The generating function $B(x)=\sum_{n\ge 1} b_nx^n$ of the Wedderburn--Etherington numbers satisfies the functional equation
\[
B(x)=x+\frac{1}{2}B(x)^2+\frac{1}{2}B(x^2), 
\] 
which was the motivation of Wedderburn in his work~\cite{Wedderburn22} {and was considered by Etherington~\cite{Etherington37}}.
This functional equation is equivalent to the recurrence relation
\[
b_{2m-1}=\sum_{i=1}^{m-1}b_ib_{2m-i-1}\quad (m\ge 2),\qquad b_{2m}= \frac{b_m(b_m+1)}{2}+\sum_{i=1}^{m-1}b_ib_{2m-i}
\]
with $b_1=1$. Using this recurrence relation, one can calculate the Wedderburn--Etherington numbers, see Table~\ref{table_bn}.

\begin{table}[H]
\begin{tabular}{c|rrrrrrrrrrrrrrr}
\toprule
$n$ & $1$ & $2$ & $3$ & $4$ & $5$ & $6$ & $7$ & $8$ & $9$ & $10$ & $11$ & $12$ & $13$ & $14$ & $15$ \\
\midrule
$b_n$ & $1$ & $1$ & $1$ & $2$ & $3$ & $6$ & $11$ & $23$ & $46$ & $98$ & $207$ & $451$ & $983$ & $2179$ & $4850$ \\ 
\bottomrule
\end{tabular}
\caption{Wedderburn--Etherington numbers $b_n$ for small values of $n$} \label{table_bn}
\end{table}

As mentioned above, {the Bruhat interval polytope} $\Q_{e,w}$ (which is the moment map image of the Schubert variety $X_{w^{-1}}$) is combinatorially equivalent to a cube if and only if $w$ is a product of distinct simple transpositions. This fact is generalized to any Lie type (see~\cite{Fan98, Karu13Schubert, LM2020}). In our forthcoming paper~{\cite{LMP_toric_Schubert},} we will discuss toric Schubert varieties in any Lie type and see that directed Dynkin diagrams appear in their classification. 

This paper is organized as follows. We illustrate the ideas underlying the paper with an example in Section~\ref{sec:example}. In Section~\ref{sec_Catalan_numbers}, we review Catalan numbers and the bijective correspondence between polygon triangulations and binary trees. In Section~\ref{sec_left_and_right_trees}, we associate primitive vectors $\mathbf{v}_{k}$'s and $\mathbf{w}_{k}$'s for $k=1,\dots,n$ with a triangulation of $\polygon{n+2}$. 
In Section~\ref{sec_lT_rT_and_binary_trees}, 
we study how these primitive vectors $\mathbf{v}_{k}$'s and $\mathbf{w}_{k}$'s are related to the binary tree associated with the polygon triangulation.
In Section~\ref{sec:fan_Catalan_type}, we form a fan using $\mathbf{v}_{k}$'s and $\mathbf{w}_{k}$'s, where these vectors become ray generators of the fan, and see when such fans are isomorphic. In Section~\ref{sec:permutations_etc}, we review how to associate a binary tree (equivalently a polygon triangulation), denoted by $\psi(u)$, with a permutation $u$. We also associate a Bruhat interval polytope with the permutation~$u$ and see that its normal fan agrees with the fan associated with $\psi(u)$. In Section~\ref{sec_smooth_toric_Richadson}, we interpret the results obtained in the previous sections in terms of Richardson varieties. We also consider products of toric Richardson varieties of Catalan type and enumerate their isomorphism classes.

\section{An example illustrating the idea} \label{sec:example}

We illustrate the idea underlying this paper with an example, which will help the reader to understand the argument developed in the paper. 

Consider two permutations $v=1243$ and $w=2431$ in $\mathfrak{S}_4$, where $v$ and $w$ are written in one-line notation. Note that the pair $(v,w)$ satisfies the condition in~\eqref{eq:simple_pair_vw}. 
The Bruhat interval $[1243, 2431]$ consists of $8$ permutations in the red part of Figure~\ref{fig:Bruhat_interval_1243-2431} and the Bruhat interval polytope $\Q_{1243,2431}$ is a $3$-cube drawn in red and thick  in Figure~\ref{fig:Bruhat_interval_polytope_1243-2431}. 
Here, for permutations $v$ and $w$ in $\mathfrak{S}_{n+1}$, the Bruhat interval $[v,w]$ is defined to be $[v,w] \coloneqq \{ z \in \mathfrak{S}_{n+1} \mid v \leq z \leq w\}$.
The entire polytope in Figure~\ref{fig:Bruhat_interval_polytope_1243-2431} is the $3$-dimensional permutohedron, where the vertices are all the permutations in $\mathfrak{S}_4$ and the label on a vertex, say $2431$, shows that the coordinate of the vertex is $(2,4,3,1)\in \R^4$.

\begin{figure}[htbp]
\begin{subfigure}[b]{0.49\textwidth}
\begin{tikzpicture}[scale=.7]
\tikzset{every node/.style={font=\footnotesize}}
\matrix [matrix of math nodes,column sep={0.52cm,between origins},
row sep={0.8cm,between origins},
nodes={circle, draw=blue!50,fill=blue!20, thick, inner sep = 0pt , minimum size=1.2mm}]
{
& & & & & \node[label = {above:{4321}}] (4321) {} ; & & & & & \\
& & &
\node[label = {above left:4312}] (4312) {} ; & &
\node[label = {above left:4231}] (4231) {} ; & &
\node[label = {above right:3421}] (3421) {} ; & & & \\
& \node[label = {above left:4132}] (4132) {} ; & &
\node[label = {left:4213}] (4213) {} ; & &
\node[label = {above:3412}] (3412) {} ; & &
\node[label = {[label distance = 0.1cm]0:2431},fill=red!20!white, draw=red!75!white] (2431) {} ; & &
\node[label = {above right:3241}] (3241) {} ; & \\
\node[label = {left:1432}, fill=red!20!white, draw=red!75!white] (1432) {} ; & &
\node[label = {left:4123}] (4123) {} ; & &
\node[label = {[label distance = 0.01cm]180:2413}, fill=red!20!white, draw=red!75!white] (2413) {} ; & &
\node[label = {[label distance = 0.01cm]0:3142}] (3142) {} ; & &
\node[label = {right:2341}, fill=red!20!white, draw=red!75!white] (2341) {} ; & &
\node[label = {right:3214}] (3214) {} ; \\
& \node[label = {below left:1423}, fill=red!20!white, draw=red!75!white] (1423) {} ; & &
\node[label = {[label distance = 0.1cm]182:1342}, fill=red!20!white, draw=red!75!white] (1342) {} ; & &
\node[label = {below:2143}, fill=red!20!white, draw=red!75!white] (2143) {} ; & &
\node[label = {right:3124}] (3124) {} ; & &
\node[label = {below right:2314}] (2314) {} ; & \\
& & & \node[label = {below left:1243}, fill=red!20!white, draw=red!75!white ] (1243) {} ; & &
\node[label = {[label distance = 0.01cm]190:1324}] (1324) {} ; & &
\node[label = {below right:2134}] (2134) {} ; & & & \\
& & & & & \node[label = {below:1234}] (1234) {} ; & & & & & \\
};
\draw (4321)--(4312)--(4132)--(1432)--(1423)--(1243)--(1234)--(2134)--(2314)--(2341)--(3241)--(3421)--(4321);
\draw (4321)--(4231)--(4132);
\draw (4231)--(3241);
\draw (4231)--(2431);
\draw (4231)--(4213);
\draw (4312)--(4213)--(2413)--(2143)--(3142)--(3241);
\draw (4312)--(3412)--(2413)--(1423)--(1324)--(1234);
\draw (3421)--(3412)--(3214)--(3124)--(1324);
\draw (3421)--(2431)--(2341)--(2143)--(2134);
\draw (4132)--(4123)--(1423);
\draw (4132)--(3142)--(3124)--(2134);
\draw (4213)--(4123)--(2143)--(1243);
\draw (4213)--(3214);
\draw (3412)--(1432)--(1342)--(1243);
\draw (2431)--(1432);
\draw (2431)--(2413)--(2314);
\draw (3142)--(1342)--(1324);
\draw (4123)--(3124);
\draw (2341)--(1342);
\draw (2314)--(1324);
\draw (3412)--(3142);
\draw (3241)--(3214)--(2314);
\draw[line width=1ex, red,nearly transparent] (1243)--(1423)--(1432)--(1342);
\draw[line width=1ex, red,nearly transparent] (1243)--(2143)--(2413)--(1423);
\draw[line width=1ex, red,nearly transparent] (2143)--(2341)--(1342);
\draw[line width=1ex, red,nearly transparent] (2413)--(2431)--(1432);
\draw[line width=1ex, red,nearly transparent] (2431)--(2341);
\draw[line width=1ex, red,nearly transparent] (1342)--(1243); 
\end{tikzpicture}
\subcaption{Bruhat interval $[1243, 2431]$} \label{fig:Bruhat_interval_1243-2431}
\end{subfigure}
\begin{subfigure}[b]{0.49\textwidth}
\tdplotsetmaincoords{-165}{120} %
\begin{tikzpicture}[scale=6, rotate=-25, tdplot_main_coords]
\tikzset{every node/.style={draw=blue!50,fill=blue!20, circle, thick, inner sep=1pt,font=\footnotesize}}


\tikzset{red node/.style = {fill=red!20!white, draw=red!75!white}}
\tikzset{red line/.style = {line width=0.4ex, red, nearly opaque}}
\coordinate (3142) at (1/3, 1/2, 1/6);
\coordinate (4231) at (2/3, 1/2, 1/6);
\coordinate (4312) at (5/6, 2/3, 1/2);
\coordinate (4321) at (5/6, 1/2, 1/3);
\coordinate (3421) at (5/6, 1/3, 1/2);
\coordinate (4213) at (2/3, 5/6, 1/2);
\coordinate (1324) at (1/3, 1/2, 5/6);
\coordinate (2413) at (2/3, 1/2, 5/6);
\coordinate (3412) at (5/6, 1/2, 2/3);
\coordinate (2314) at (1/2, 2/3, 5/6);
\coordinate (4123) at (1/2, 5/6, 1/3);
\coordinate (4132) at (1/2, 2/3, 1/6);
\coordinate (3214) at (1/2, 5/6, 2/3);
\coordinate (3124) at (1/3, 5/6, 1/2);
\coordinate (2431) at (2/3, 1/6, 1/2);
\coordinate (1432) at (1/2, 1/6, 2/3);
\coordinate (1423) at (1/2, 1/3, 5/6);
\coordinate (1342) at (1/3, 1/6, 1/2);
\coordinate (2341) at (1/2, 1/6, 1/3);
\coordinate (3241) at (1/2, 1/3, 1/6);
\coordinate (1243) at (1/6, 1/3, 1/2);
\coordinate (2143) at (1/6, 1/2, 1/3);
\coordinate (1234) at (1/6, 1/2, 2/3);
\coordinate (2134) at (1/6, 2/3, 1/2);
\draw[thick, draw=blue!70, dashed] (2134)--(2143)--(3142)--(4132)--(4123);
\draw[thick, draw=blue!70, dashed] (2143)--(1243);
\draw[thick, draw=blue!70, dashed] (3142)--(3241)--(2341)--(1342);
\draw[thick, draw=blue!70, dashed] (2341)--(2431);
\draw[thick, draw=blue!70, dashed] (3241)--(4231)--(4132);
\draw[thick, draw=blue!70, dashed] (4231)--(4321);
\draw[red line] (1243)--(2143)--(2413);
\draw[thick, draw=blue!70] (4213)--(4312)--(3412)--(2413)--(2314)--(3214)--cycle;
\draw[thick, draw=blue!70] (4312)--(4321)--(3421)--(3412);
\draw[thick, draw=blue!70] (3421)--(2431)--(1432)--(1423)--(2413);
\draw[thick, draw=blue!70] (1423)--(1324)--(2314);
\draw[thick, draw=blue!70] (1432)--(1342)--(1243)--(1234)--(1324);
\draw[thick, draw=blue!70] (1234)--(2134)--(3124)--(3214);
\draw[thick, draw=blue!70] (3124)--(4123)--(4213);
\draw[red line, dashed] (2143)--(2341)--(2431);
\draw[red line, dashed] (2341)--(1342);
\draw[red line] (1423)--(1432)--(2431)--(2413)--cycle;
\draw[red line] (1423)--(1243);
\draw[red line] (1243)--(1342)--(1432);
\node [label = {[label distance = -0.2cm]below:1234}] at (1234) {};
\node[label = {[label distance = 0cm]left:1243}, red node] at (1243) {};
\node[label = {[label distance = 0cm]left:1324}] at (1324) {};
\node[label = {[label distance = 0cm]left:1342}, red node] at (1342) {};
\node [label = {[label distance = 0cm]right:1423},red node] at (1423) {};
\node[label = {[label distance = 0cm]left:1432}, red node] at (1432) {};
\node [label = {[label distance = 0cm]right:2134}] at (2134) {};
\node[label = {[label distance = -0.1cm]above right:2143}, red node] at (2143) {};
\node[label = {[label distance = 0cm]right:2314}] at (2314) {};
\node[label = {[label distance = 0cm]right:2341}, red node] at (2341) {};
\node[label = {[label distance = -0.1cm]right :2413}, red node] at (2413) {};
\node[label = {[label distance = 0cm]left:2431}, red node] at (2431) {};
\node[label = {[label distance = 0cm]right:3124}] at (3124) {};
\node[label = {[label distance = -0.2cm]above:3142}] at (3142) {};
\node[label = {[label distance = 0cm]right:3214}] at (3214) {};
\node [label = {[label distance = -0.1cm]above:3241}] at (3241) {};
\node[label = {[label distance = 0cm]below left:3412}] at (3412) {};
\node[label = {[label distance = 0cm]above:3421}] at (3421) {};
\node[label = {[label distance = 0cm]right:4123}] at (4123) {};
\node [label = {[label distance = -0.2cm]below:4132}] at (4132) {};
\node[label = {[label distance = 0cm]right:4213}] at (4213) {};
\node[label = {[label distance = 0cm]right:4231}] at (4231) {};
\node[label = {[label distance = 0cm]right:4312}] at (4312) {};
\node [label = {[label distance = -0.2cm]above:4321}] at (4321) {};
\end{tikzpicture}
\subcaption{ Bruhat interval polytope $\Q_{1243, 2431}$} \label{fig:Bruhat_interval_polytope_1243-2431}
\end{subfigure}
\caption{A Bruhat interval polytope which is a 3-cube.}\label{fig:1243-2431}
\end{figure}
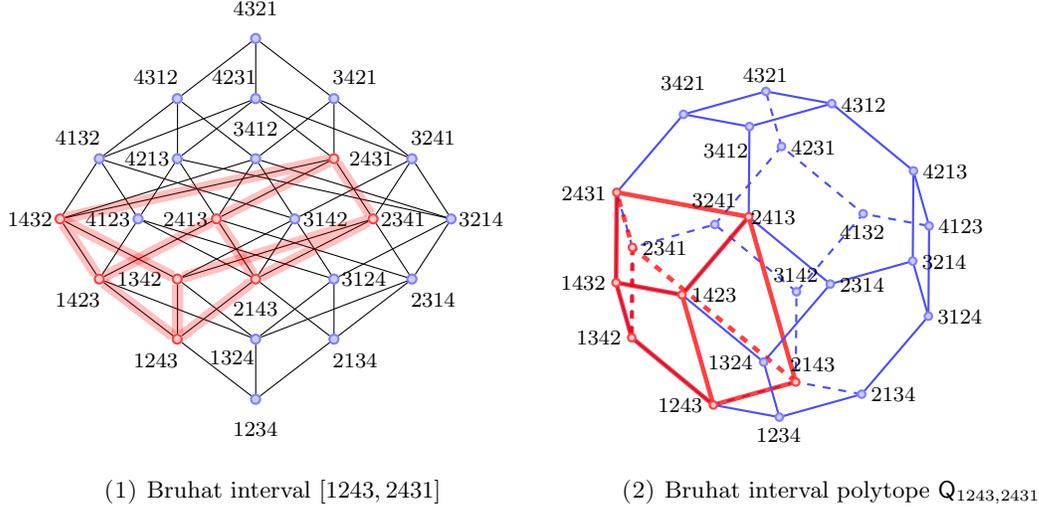

One sees from Figure~\ref{fig:Bruhat_interval_polytope_1243-2431} that there are three edges emanating from the vertex $v=1243$ (resp. $w=2431$) and their primitive edge vectors are 
\begin{equation} \label{eq:example_atom_coatom}
\begin{split}
&\mathbf{p}_1=e_1-e_2,\quad \mathbf{p}_2=e_2-e_3,\quad \mathbf{p}_3=e_2-e_4 \\
(\text{resp. } &\mathbf{q}_1=-e_1+e_4,\ \mathbf{q}_2=-e_2+e_3,\ \mathbf{q}_3=-e_3+e_4),
\end{split}
\end{equation}
where $e_1,e_2,e_3,e_4$ denote the standard basis of $\R^4$. These primitive 
vectors correspond to the atoms and coatoms of the Bruhat interval 
$[1243,2431]$. More precisely, the following pairs $\{i,j\}$ 
\begin{equation} \label{eq:pairs_s4}
\{1,2\},\ \{2,3\},\ \{2,4\} \quad (\text{resp. } \{1,4\},\ \{2,3\},\ \{3,4\}). 
\end{equation}
satisfy that $vt_{i,j}$ covers $v$ and $vt_{i,j}\le w$ (resp. 
$wt_{i,j}$ is covered by $w$ and $v\le wt_{i,j}$), where $t_{i,j}$ denotes the 
transposition interchanging $i$ and $j$.\footnote{For permutations $x$ and 
$y$, we 
say $y$ \emph{covers} $x$ (or equivalently, $x$ \emph{is covered by} $y$) if 
there does not exist $z$ such that  $x < z < y$. }
These correspond to the primitive vectors in~\eqref{eq:example_atom_coatom}. 
We subtract $1$ from the first three pairs above (corresponding to the atoms) in each element, so that we obtain 
\begin{equation} \label{eq:modified_pairs}
\{0,1\},\ \{1,2\},\ \{1,3\} \quad (\text{resp. } \{1,4\},\ \{2,3\},\ \{3,4\})
\end{equation}
and we may regard them as edges or diagonals of the pentagon $\polygon{5}$ 
with vertices labelled from $0$ to $4$ in counterclockwise order. The result is 
shown in Figure~\ref{fig:edge_vectors}, where edges or diagonals of 
$\polygon{5}$ obtained from the first three pairs in \eqref{eq:modified_pairs} 
are shown by blue solid lines while those obtained from the latter three pairs are 
shown by red dashed lines. 
They form a triangulation of $\polygon{5}$.
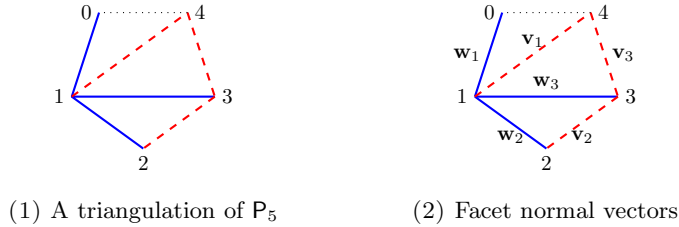
\begin{figure}[htbp]
\begin{subfigure}[b]{0.35\textwidth}
\begin{center}
\begin{tikzpicture}[x=10mm, y=5mm, label distance = 1mm, every node/.style={scale=0.8}]
\foreach \x in {2,...,5,6}{
\coordinate (a\x) at (-18-\x*72:1cm) ; 
} 
\draw[blue, thick] (a3)--(a2)
(a2)--(a6)
(a2)--(a5);
\draw[red, dashed, thick] (a4)--(a5)
(a2)--(a4)
(a5)--(a6); 

\node[right] at (a4) {$4$};
\node[right] at (a5) {$3$};
\node[below] at (a6) {$2$};
\node[left] at (a2) {$1$};
\node[left] at (a3) {$0$};
\draw[dotted] (a4)--(a3);
\end{tikzpicture}
\end{center}
\caption{A triangulation of $\polygon{5}$}\label{fig:edge_vectors}
\end{subfigure} 
\begin{subfigure}[b]{0.35\textwidth}
\begin{center}
\begin{tikzpicture}[x=10mm, y=5mm, label distance = 1mm, every node/.style={scale=0.8}]
\foreach \x in {2,...,5,6}{
\coordinate (a\x) at (-18-\x*72:1cm) ; 
} 
\draw[blue, thick] (a3)--(a2) node[left, midway, black]{$\mathbf{w}_1$};
\draw[blue, thick] (a2)--(a6) node[below, midway, black]{$\mathbf{w}_2$};
\draw[blue, thick] (a2)--(a5) node[above, midway, black]{$\mathbf{w}_3$};
\draw[red, dashed, thick] (a4)--(a5) node[right, midway, black]{$\mathbf{v}_3$};
\draw[red, dashed, thick] (a2)--(a4) node[above, midway, black]{$\mathbf{v}_1$};
\draw[red, dashed, thick] (a5)--(a6) node[below, midway, black]{$\mathbf{v}_2$};
\node[right] at (a4) {$4$};
\node[right] at (a5) {$3$};
\node[below] at (a6) {$2$};
\node[left] at (a2) {$1$};
\node[left] at (a3) {$0$};
\draw[dotted] (a4)--(a3);
\end{tikzpicture}
\end{center}
\caption{Facet normal vectors}\label{fig:normal_vectors}
\end{subfigure} 
\caption{Triangulation and facet normal vectors obtained from $\Q_{1243,2431}$}
\label{fig:example_pentagon}
\end{figure}

The reader may wonder why we subtract $1$ from the first three pairs in 
\eqref{eq:pairs_s4} corresponding to the atoms. 
Even though the reason will be revealed in Section~\ref{sec:permutations_etc} (see Proposition~\ref{prop_u_head_atoms_and_lTrT}), we briefly explain it.
The 
permutation $v=1243$ (resp. $w=2431$) is obtained from $243$ by putting the 
number $1$ at the head (resp. at the tail). If we regard the positions of the 
numbers $2, 4, 3$ in $243$ as the $1$st, the $2$nd, the $3$rd in this order, 
then the positions of the numbers $1, 2, 4, 3$ in $v=1243$ are the $0$th, the  
$1$st, the $2$nd, the $3$rd in this order. In this regard, a permutation which 
covers $v$, say $2143$, is obtained from $v=1243$ by interchanging the $0$th 
position and the $1$st position, so that we obtain the pair $\{0,1\}$ 
in~\eqref{eq:modified_pairs}. This is the reason why we subtract $1$ from the 
pairs in~\eqref{eq:pairs_s4} corresponding to the atoms but leave the pairs 
corresponding to the coatoms unchanged. 

The primitive edge vectors $\mathbf{p}_1, \mathbf{p}_2,\mathbf{p}_3$ (resp. $\mathbf{q}_1, \mathbf{q}_2, \mathbf{q}_3$) in \eqref{eq:example_atom_coatom} form a basis of the sublattice~$M$ of $\Z^4$ with the sum of the coordinates equal to zero. Through the dot product on $\Z^4$, we can think of the dual lattice of $M$ as the quotient lattice $N=\Z^4/\Z(1,1,1,1)$. Let $\varpi_k$ $(k=0,1,2,3,4)$ be the quotient image of $\sum_{i=1}^k e_i$ in $N$. Then $\{\varpi_1,\dots,\varpi_4\}$ is a basis of $N$ and $\varpi_0=\varpi_{5}=\boldsymbol{0}$ by definition.
The dual basis of $\mathbf{p}_1, \mathbf{p}_2,\mathbf{p}_3$ (resp. $\mathbf{q}_1, \mathbf{q}_2, \mathbf{q}_3$) is given by 
\begin{equation} \label{eq:example_facet_vw}
\begin{split}
&\mathbf{v}_1=\varpi_1=\varpi_1-\varpi_4,\quad\ \mathbf{v}_2=\varpi_2-\varpi_3,\ \ \mathbf{v}_3=\varpi_3-\varpi_4,\\
(\text{resp. }& \mathbf{w}_1=-\varpi_1=\varpi_0-\varpi_1,\ \mathbf{w}_2=\varpi_1-\varpi_2,\ \mathbf{w}_3=\varpi_1-\varpi_3). 
\end{split}
\end{equation}
These vectors may be regarded as the primitive inward facet normal vectors of the $3$-cube $\Q_{1243,2431}$, so they are ray generators of the normal fan of $\Q_{1243,2431}$. Finally, we note that $\mathbf{v}_1, \mathbf{v}_2, \mathbf{v}_3, \mathbf{w}_1, \mathbf{w}_2, \mathbf{w}_3$ in \eqref{eq:example_facet_vw} can be assigned to edges or diagonals of $\polygon{5}$ as shown in Figure~\ref{fig:normal_vectors} by looking at the suffixes of $\varpi_i$'s in their expression \eqref{eq:example_facet_vw}.  Then the three relations 
\[
\mathbf{v}_1+\mathbf{w}_1=\mathbf{0}, \quad \mathbf{v}_2+\mathbf{w}_2=\mathbf{w}_3, \quad \mathbf{v}_3+\mathbf{w}_3=\mathbf{v}_1
\]
obtained from \eqref{eq:example_facet_vw} correspond to the three triangles in our triangulation of $\polygon{5}$ as is seen in Figure~\ref{fig:normal_vectors}, where we understand that the zero vector $\mathbf{0}$ is assigned to the distinguished edge connecting the vertices $0$ and $4$. 

\section{Catalan numbers: polygon triangulations and binary trees}\label{sec_Catalan_numbers}
There are several equivalent definitions of \textit{Catalan numbers} $C_n=\frac{1}{n+1}\binom{2n}{n}$ as is provided in~\cite{Stanley_Catalan}. 
In this section, we recall two combinatorial models of Catalan numbers: triangulations of polygons and binary trees. 

We first recall the polygon triangulation model. Let $\polygon{n+2}$ denote a convex polygon in the plane with $n+2$ vertices (or \textit{convex $(n+2)$-gon} for simplicity).
We label the vertices from $0$ to $n+1$ in counterclockwise order.
A \textit{triangulation} of $\polygon{n+2}$ is a decomposition of $\polygon{n+2}$ into a set of $n$ triangles by adding $n-1$ diagonals of $\polygon{n+2}$ which do not intersect in their interiors. We also mean by a triangulation of $\polygon{n+2}$ the set of the boundary edges of the $n$ triangles.\footnote{In~\cite{Stanley_Catalan}, a \textit{triangulation} of $\polygon{n+2}$ is defined to be a set of $n-1$ diagonals of $\polygon{n+2}$ which do not intersect in their interiors.}  For example, we present triangulations of $\polygon{5}$ in Figure~\ref{fig_triangulation_5gon}. The number of triangulations of $\polygon{n+2}$ is known to be the Catalan number $C_n$.

\begin{figure}[hbtp]
	\begin{tikzpicture}[x=10mm, y=5mm, label distance = 1mm]
	\foreach \x in {2,...,5,6}{
		\coordinate  (a\x)   at (-18-\x*72:1cm)  ;	
	} 

\draw[thick] (a2)--(a3)--(a4)--(a5)--(a6)--(a2);
\draw[thick] (a2)--(a4)
(a4)--(a6);

\node[right] at (a4) {$4$};
\node[right] at (a5) {$3$};
\node[below] at (a6) {$2$};
\node[left] at (a2) {$1$};
\node[left] at (a3) {$0$};
	\end{tikzpicture}\hspace{0.1cm} %
		\begin{tikzpicture}[x=10mm, y=5mm, label distance = 1mm]
	\foreach \x in {2,...,5,6}{
		\coordinate  (a\x)   at (-18-\x*72:1cm)  ;	
	} 
	
	\draw[thick] (a2)--(a3)--(a4)--(a5)--(a6)--(a2);
	\draw[thick] (a4)--(a2)
	(a5)--(a2);
	
\node[right] at (a4) {$4$};
\node[right] at (a5) {$3$};
\node[below] at (a6) {$2$};
\node[left] at (a2) {$1$};
\node[left] at (a3) {$0$};
	\end{tikzpicture}%
	\hspace{0.1cm} %
	\begin{tikzpicture}[x=10mm, y=5mm, label distance = 1mm]
	\foreach \x in {2,...,5,6}{
		\coordinate  (a\x)   at (-18-\x*72:1cm)  ;	
	} 
	
	\draw[thick] (a2)--(a3)--(a4)--(a5)--(a6)--(a2);
	\draw[thick] (a5)--(a3)
	(a5)--(a2);
	
\node[right] at (a4) {$4$};
\node[right] at (a5) {$3$};
\node[below] at (a6) {$2$};
\node[left] at (a2) {$1$};
\node[left] at (a3) {$0$};
	\end{tikzpicture}		\hspace{0.1cm} %
	\begin{tikzpicture}[x=10mm, y=5mm, label distance = 1mm]
	\foreach \x in {2,...,5,6}{
		\coordinate  (a\x)   at (-18-\x*72:1cm)  ;	
	} 
	
	\draw[thick] (a2)--(a3)--(a4)--(a5)--(a6)--(a2);
	\draw[thick] (a5)--(a3)
	(a3)--(a6);
	
\node[right] at (a4) {$4$};
\node[right] at (a5) {$3$};
\node[below] at (a6) {$2$};
\node[left] at (a2) {$1$};
\node[left] at (a3) {$0$};
	\end{tikzpicture}%
	\hspace{0.1cm} %
	\begin{tikzpicture}[x=10mm, y=5mm, label distance = 1mm]
	\foreach \x in {2,...,5,6}{
		\coordinate  (a\x)   at (-18-\x*72:1cm)  ;	
	} 
	
	\draw[thick] (a2)--(a3)--(a4)--(a5)--(a6)--(a2);
	\draw[thick] (a4)--(a6)
	(a3)--(a6);
	
\node[right] at (a4) {$4$};
\node[right] at (a5) {$3$};
\node[below] at (a6) {$2$};
\node[left] at (a2) {$1$};
\node[left] at (a3) {$0$};
	\end{tikzpicture}
	\caption{Triangulations of $\polygon{5}$.}\label{fig_triangulation_5gon}
\end{figure}

The \textit{binary trees} provide another combinatorial model for Catalan numbers (see~\cite[Theorem~1.5.1]{Stanley_Catalan}). 
A binary tree is defined recursively as follows. The empty set $\emptyset$ is a binary tree. Otherwise, a binary tree has a \textit{root vertex} $v$, a \textit{left  subtree} $\B_1$, and a \textit{right subtree} $\B_2$, both of which are binary trees. 
We draw a binary tree by putting the root vertex $v$ at the top, the left subtree $\B_1$ below and to the left of $v$, and the right subtree $\B_2$ below and to the right of $v$. Moreover, we draw an edge from $v$ to the root of each nonempty $\B_i$. 
Hence, each vertex of a binary tree is connected to at most two children, which are called the \textit{left child} and the \textit{right child} (one or both can be empty).
See Figure~\ref{fig_binary_tree_v3} for binary trees with three vertices. We draw additional circles to decorate the root vertices.
\begin{figure}[h]
		\begin{tikzpicture}[every node/.style = {circle, fill=black,  inner sep = 1.5pt, },
	level distance=5mm,	
	level 1/.style={sibling distance=7mm},	
	level 2/.style={sibling distance=5mm},
	level 3/.style={sibling distance=5mm}	
	]
	
	\node[black, draw, circle, inner sep = 1.5 pt, fill=black, double]  {}
		child[missing] {}
		child{
			node {}
			child[missing] {}
			child{
				node {}
		}
};
	\end{tikzpicture}\hspace{1cm}%
		\begin{tikzpicture}[every node/.style = {circle, fill=black,  inner sep = 1.5pt, },
	level distance=5mm,	
	level 1/.style={sibling distance=7mm},	
	level 2/.style={sibling distance=5mm},
	level 3/.style={sibling distance=5mm}	
	]
	
	\node[black, draw, circle, inner sep = 1.5 pt, fill=black, double]  {}
	child[missing] {}
	child{
		node {}
		child{
			node {}
		}
		child[missing] {}	
	};
	\end{tikzpicture}%
\hspace{1cm}%
	\begin{tikzpicture}[every node/.style = {circle, fill=black,  inner sep = 1.5pt, },
	level distance=5mm,	
	level 1/.style={sibling distance=7mm},	
	level 2/.style={sibling distance=5mm},
	level 3/.style={sibling distance=5mm}	
	]
	
	\node[black, draw, circle, inner sep = 1.5 pt, fill=black, double]  {}
	child{
		node {}
		child[missing] {}
		child{
			node {}
		}
	}	
	child[missing] {}	;
	\end{tikzpicture}\hspace{1cm}%
	\begin{tikzpicture}[every node/.style = {circle, fill=black,  inner sep = 1.5pt, },
	level distance=5mm,	
	level 1/.style={sibling distance=7mm},	
	level 2/.style={sibling distance=5mm},
	level 3/.style={sibling distance=5mm}	
	]
	
	\node[black, draw, circle, inner sep = 1.5 pt, fill=black, double]  {}
	child{
		node {}
		child{
			node {}
		}
		child[missing] {}	
	}	
	child[missing] {}	;
	\end{tikzpicture}\hspace{1cm}%
\raisebox{1.5em}{	\begin{tikzpicture}[every node/.style = {circle, fill=black,  inner sep = 1.5pt, },
	level distance=5mm,	
	level 1/.style={sibling distance=7mm},	
	level 2/.style={sibling distance=5mm},
	level 3/.style={sibling distance=5mm}	
	]
	
	\node[black, draw, circle, inner sep = 1.5 pt, fill=black, double]  {}
	child{
		node {}
	}	
	child{
		node{}
	}	;
	\end{tikzpicture}}
\caption{Binary trees with three vertices}\label{fig_binary_tree_v3}
\end{figure}
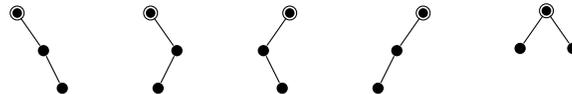

As we already mentioned, there is a bijective correspondence between the set of triangulations~$T$ of a convex polygon $\polygon{n+2}$ with $n+2$ vertices and that of binary trees $\B_T$ with $n$ vertices. We recall the construction of a bijection from triangulations $T$ of $\polygon{n+2}$ to binary trees $\B_T$. 
Let $e$ be an edge connecting $0$ and $n+1$. (Note that one may choose any edge of the polygon to get a bijective correspondence.) We associate a vertex of 
$\B_T$ with a triangle in $T$. The root vertex $v$ of~$\B_T$ corresponds to the triangle of which $e$ is an edge. 
Here, we notice that for the boundary edge~$e$ of $\polygon{n+2}$, there uniquely exists a triangle containing $e$ in the triangulation $T$ of $\polygon{n+2}$. Hence, the root vertex of $\B_T$ is well-defined.

For the triangle containing the edge $e$, we denote the remaining edges by $f_{\ll}(e)$ and $f_{\rr}(e)$ such that $f_{\ll}(e)$, $f_{\rr}(e)$, and $e$ are placed  counterclockwise. 
The vertex associated with the other triangle containing ~$f_{\ll}(e)$ becomes the left child of the root vertex $v$ of $\B_T$, and similarly, the other triangle containing $f_{\rr}(e)$ defines the right child of  the root vertex $v$ of $\B_T$. Here, if the triangle corresponding to $v$ is the only triangle containing~$f_{\ll}(e)$ (i.e., $f_{\ll}$(e) is a boundary edge), then the left subtree of $v$ is the empty set. Similarly, we also set the right subtree of $v$ is the empty set if $f_{\rr}(e)$ is a boundary edge.
Continuing this process, we obtain a binary tree $\B_T$. 
In Figure~\ref{fig_tree_and_triangulation}, we draw a binary tree associated with a triangulation, where the triangle containing the edge $e$ is colored in yellow and the root of $\B_T$ is decorated with an additional circle.

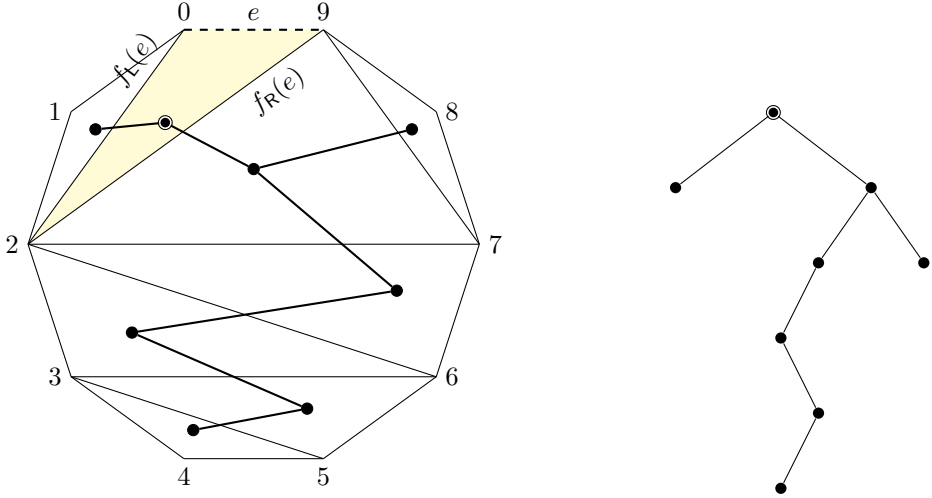
\begin{figure}[h]
\begin{center}
	\begin{tikzpicture}
\foreach \x in {2,...,11}{
	\coordinate (\x)  at (3*36-\x*36:3cm) ;	
}

\node[above] at (11) {$9$};
\foreach \x/\y in {2/8,3/7,4/6}{
	\node[right] at (\x) {$\y$};
}
\foreach \x/\y in {5/5,6/4}{
	\node[below] at (\x) {$\y$};
}
\foreach \x/\y in {7/3,8/2,9/1}{
	\node[left] at (\x) {$\y$};
}
\node[above] at (10) {$0$};

\fill[yellow!20!white] (10)--(11)--(8);

\draw[thick, dashed] (10)--(11) node[above, midway] {$e$};
\draw (11)--(8) node[below, pos = 0.2, sloped] {$f_{\rr}(e)$};
\draw
(11) edge (2)
(11) edge (3)
(3) edge (4)
(5) edge (6)
(8) edge (9)
(4) edge (5)
(4) edge (7);
\draw  (10) -- (8) node[above, pos=0.2, sloped] {$f_{\ll}(e)$};
\draw
(10) edge (9)
(8) edge (7)
(8) edge (4)
(8) edge (3)
(7) edge (6)
(7) edge (5)
(3) edge (2) ;
\node[black, draw, circle, inner sep = 1.5 pt, fill=black, double]  (v1) at (-6.5*36:2cm) {};
\node[black, draw, circle, inner sep = 1.5pt, fill=black]  (v2) at (-6*36:2.6cm) {};
\node[black, draw, circle, inner sep = 1.5pt, fill=black]  (v3) at (2.5*36:1cm) {}; 
\node[black, draw, circle, inner sep = 1.5pt, fill=black]  (v4) at (1*36:2.6cm) {};
\node[black, draw, circle, inner sep = 1.5pt, fill=black]  (v5) at (-0.5*36:2cm) {};
\node[black, draw, circle, inner sep = 1.5pt, fill=black]  (v6) at (-4*36:2cm) {};
\node[black, draw, circle, inner sep = 1.5pt, fill=black]  (v7) at (-2*36:2.3cm) {};
\node[black, draw, circle, inner sep = 1.5pt, fill=black]  (v8) at (-3*36:2.6cm) {};	

\draw[black, thick] (v1)--(v2)
(v1)--(v3)
(v3)--(v4)
(v3)--(v5)
(v5)--(v6)
(v6)--(v7)
(v7)--(v8);	 %
\end{tikzpicture} \hspace{2cm}%
\begin{tikzpicture}[every node/.style = {circle, fill=black,  inner sep = 1.5pt, },
		level distance=5mm,	
		level 1/.style={sibling distance=13mm},	
		level 2/.style={sibling distance=7mm},
		level 3/.style={sibling distance=5mm},
		scale=2	
		]
		\node[black, draw, circle, inner sep = 1.5 pt, fill=black, double]{}
		child{
			node{}
		}
		child{
			node{}
			child{
				node{}
				child{
					node{}
					child[missing] {}
					child{
						node{}
						child{
							node{}
						}
						child[missing] {}
					}
				}
				child[missing] {}
			}
			child{
				node{}
			}
		};
		
		\end{tikzpicture}
\end{center}
\caption{A triangulated polygon and the associated binary tree.}\label{fig_tree_and_triangulation}
\end{figure}

\section{Left and right trees of polygon triangulations}\label{sec_left_and_right_trees}

In this section, we introduce the \textit{left tree} and the \textit{right tree} of a polygon triangulation~$T$. Moreover, we define vectors obtained from the left and right trees, and study their properties.
 
Take the distinguished edge $e$ connecting $0$ and $n+1$. 
We inductively define a map \linebreak $F \colon T \setminus \{e\} \to \{\ll, \rr\}$ as follows, where the triangulation $T$ is regarded as the set of the boundary edges of the $n$ triangles.
For the triangle containing the edge~$e$, we denote the remaining edges by $f_{\ll}(e)$ and $f_{\rr}(e)$ such that $f_{\ll}(e)$, $f_{\rr}(e)$, and $e$ are placed counterclockwise. Then we set $F(f_{\ll}(e)) = \ll$ and $F(f_{\rr}(e)) = \rr$. 
When we remove the edge~$e$ from~$T$, we obtain two triangulated polygons, say $Q_{\ll}$ and ${Q}_{\rr}$, with one common vertex.\footnote{If one of $Q_{\ll}$ and $Q_{\rr}$, say $Q_{\bullet}$, is a single edge, then we do not apply the following process to $Q_{\bullet}$.} 
For each triangulation, we continue the same process. That is, for the triangulation $Q_\bullet$ having the distinguished edge $f_\bullet(e)$, we consider the triangle containing $f_\bullet(e)$, and denote the remaining edges of the triangle by $f_{\ll}(f_\bullet(e))$ and $f_{\rr}(f_\bullet(e))$ such that $f_{\ll}(f_{\bullet}(e))$, $f_{\rr}(f_{\bullet}(e))$, and $f_{\bullet}(e)$ are placed counterclockwise for each $\bullet =\ll,\rr$. 
Then we define $F(f_{\ll}(f_{\bullet}(e))) = \ll$ and $F(f_{\rr}(f_{\bullet}(e))) = \rr$. Continuing this process, we obtain a map $F \colon T \setminus \{e\} \to \{\ll, \rr\}$. Indeed, the map produces a partition of the set $T \setminus \{e\}$.

\begin{definition}
	Let $T$ be a triangulation of $\polygon{n+2}$.
We define two rooted graphs, called the \textit{left graph} $\lT$ and the \textit{right graph}~$\rT$ of~$T$ as follows.  
\begin{itemize}
	\item $V(\lT) = {\{0,1,\dots,n\}}$ and $0$ is the root vertex; $E(\lT) = \{ f \in T \mid F(f) = \ll \}$.
	\item 	$V(\rT) ={\{1,2,\dots,n+1\}}$ and $n+1$ is the root vertex; $E(\rT) = \{f \in T \mid F(f) = \rr\}$.
\end{itemize}
\end{definition}

\begin{example}\label{example_LR_tree_from_triangulation}
	Let $T$ be a triangulation of $\polygon{10}$ given in Figure~\ref{fig_tree_and_triangulation}, which is the following set of edges:
	\[
	T = \{\{i, i+1\} \mid 0 \leq i \leq 8\} \cup \{ 0,9\} \cup \{ \{ 0,2\}, \{2,6\}, \{2,7\}, \{2,9\}, \{3,6\}, \{3,5\}, \{7,9\}\}.
	\]
By taking the edge $e  = \{0,9\}$, we get $f_{\ll}(e) = \{0,2\}$ and $f_{\rr}(e) = \{ 2,9\}$, so we have $F(\{0,2\}) = \ll$ and $F(\{2,9\}) = \rr$. By removing the edge $e$, we get a triangulation ${Q}_{\ll}$ of $\polygon{3}$ and a triangulation~${Q}_{\rr}$ of~$\polygon{8}$. For the edge $f_{\ll}(e) = \{ 0,2\}$, we have $f_{\ll}(\{0,2\}) = \{0,1\}$ and $f_{\rr}(\{0,2\}) = \{1,2\}$. On the other hand, for the edge $f_{\rr}(e) = \{2,9\}$, we have $f_{\ll}(\{2,9\}) = \{2,7\}$ and $f_{\rr}(\{2,9\}) = \{7,9\}$. Accordingly, we get 
\[
F(\{0,1\}) = F(\{2,7\}) = \ll, \quad F(\{1,2\}) = F(\{7,9\}) = \rr.
\]
Continuing this process, we obtain the following. 
\begin{center}
\begin{tabular}{l|cccccccc}
\toprule 
$f$ & $\{0,1\}$ &$\{ 0,2\}$ & $\{2,3\}$ & $\{3,4\}$ &$\{3,5\}$   & $\{2,6\}$ &$ \{2,7\} $
& $\{7,8\}$ \\
$F(f)$& $\ll$& $\ll$ & $\ll$ & $\ll$ & $\ll$ & $\ll$ & $\ll$ & $\ll$\\
\midrule
$f$ & $\{1,2\}$  &$\{2,9\}$ &$\{3,6\}$ & $\{4,5\}$  & $\{5,6\}$  & $\{6,7\}$ &$ \{7,9\}$ 
& $\{8,9\}$  \\
$F(f)$  & $\rr$  & $\rr$ & $\rr$  & $\rr$ & $\rr$  & $\rr$ & $\rr$  & $\rr$ \\
\bottomrule
\end{tabular}
\end{center}

\medskip

We depict the left graph and the right graph of this triangulation $T$ of $\polygon{10}$ in Figure~\ref{fig_atom_and_coatom_together}.
The left graph $\lT$ is colored in blue while the right graph $\rT$ is colored in red and dashed.
\end{example}
\begin{remark}
	There is a bijective correspondence between the set of triangulations of $\polygon{n+2}$ and that of full binary trees with $2n+1$ vertices (or $n+1$ endpoints) as shown in~\cite[{\#5 in \S2}]{Stanley_Catalan}. Here, we say a binary tree is \emph{full} if every vertex has zero or two children. 
	For a triangulation~$T$ of~$\polygon{n+2}$, the corresponding full binary tree $\mathcal{C}_T$ is obtained by adding leaves to $\mathcal{B}_T$ as follows. Recall the construction of~$\mathcal{B}_T$ that the left (resp. right) subtree of a vertex $v$ becomes the empty set whenever the triangle (in $T$) corresponding to $v$ is formed by edges $f_{\ll}(f), f_{\rr}(f), f$ and the edge $f_{\ll}(f)$ (resp. $f_{\rr}(f)$) is on the boundary of $\polygon{n+2}$. 
To construct $\mathcal{C}_T$, we add the left (resp. right) leaf vertex of a vertex $v$ of $\mathcal{B}_T$ if the edge $f_{\ll}(f)$ (resp. $f_{\rr}(f)$) is on the boundary of $\polygon{n+2}$, where the edges $f_{\ll}(f), f_{\rr}(f), f$ form the triangle (in $T$) corresponding to $v$. 
By drawing this full binary tree~$\mathcal{C}_T$ on the triangulation~$T$, one can see that the edges of the left graph $\lT$ (resp. the right graph~$\rT$) intersect the edges of $\mathcal{C}_T$ connecting a vertex and its left child (resp. its right child). See Figure~\ref{fig_atom_and_coatom_together}.
\end{remark}
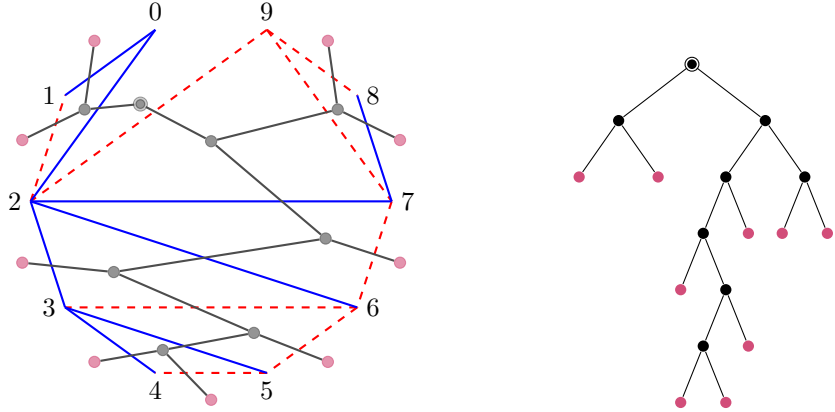
\begin{figure}[hbt]
	\begin{center}
		\begin{tikzpicture}[x=10mm, y=5mm, label distance = 1mm, scale = 0.8]
\foreach \x in {1,2,...,11}{
	\coordinate (\x)  at (3*36-\x*36:3cm) ;	
}

\node[above] at (11) {$9$};
\foreach \x/\y in {2/8,3/7,4/6}{
	\node[right] at (\x) {$\y$};
}
\foreach \x/\y in {5/5,6/4}{
	\node[below] at (\x) {$\y$};
}
\foreach \x/\y in {7/3,8/2,9/1}{
	\node[left] at (\x) {$\y$};
}
\node[above] at (10) {$0$};

		\draw[red, dashed, thick ] 
		(1) edge (2) 
		(1) edge (3)
		(1) edge (8)
		(3) edge (4)
		(5) edge (6)
		(8) edge (9)
		(4) edge (5)
		(4) edge (7);
		
		\draw[blue, thick] (10) edge (8)
		(10) edge (9)
		(8) edge (7)
		(8) edge (4)
		(8) edge (3)
		(7) edge (6)
		(7) edge (5)
		(3) edge (2) ;
		
\begin{scope}[every node/.style={opacity=0.6}]
		\node[black!70, draw, circle, inner sep = 1.5 pt, fill=black!70, double]  (v1) at (-6.5*36:2cm) {};
		\node[black!70, draw, circle, inner sep = 1.5pt, fill=black!70]  (v2) at (-6*36:2.6cm) {};
		\node[black!70, draw, circle, inner sep = 1.5pt, fill=black!70]  (v3) at (2.5*36:1cm) {}; 
		\node[black!70, draw, circle, inner sep = 1.5pt, fill=black!70]  (v4) at (1*36:2.6cm) {};
		\node[black!70, draw, circle, inner sep = 1.5pt, fill=black!70]  (v5) at (-0.5*36:2cm) {};
		\node[black!70, draw, circle, inner sep = 1.5pt, fill=black!70]  (v6) at (-4*36:2cm) {};
		\node[black!70, draw, circle, inner sep = 1.5pt, fill=black!70]  (v7) at (-2*36:2.3cm) {};
		\node[black!70, draw, circle, inner sep = 1.5pt, fill=black!70]  (v8) at (-3*36:2.6cm) {};	
		
	\foreach \x/\y in { 2/8, 3/7, 4/6, 5/5, 6/4, 7/3, 8/2, 9/1, 10/0}{
	\node[purple!70, draw, circle, inner sep = 1.5pt, fill=purple!70]  (vv\x)  at (3*36-\x*36+18:3.3cm)  {};	
} 		
		
		\draw[black!70, thick] (v1)--(v2)
		(v1)--(v3)
		(v3)--(v4)
		(v3)--(v5)
		(v5)--(v6)
		(v6)--(v7)
		(v7)--(v8);	 %
		\draw[black!70, thick] (v2)--(vv10)
		(v2)--(vv9)
		(v4)--(vv2)
		(v4)--(vv3)
		(v5)--(vv4)
		(v6)--(vv8)
		(v7)--(vv5)
		(v8)--(vv6)
		(v8)--(vv7);
\end{scope}
		\end{tikzpicture}\hspace{2cm}%
		\begin{tikzpicture}[nb/.style = {circle, fill=black,  inner sep = 1.5pt, },
		nc/.style = {circle, fill=purple!70,  inner sep = 1.5pt, },
		level distance=5mm,	
		level 1/.style={sibling distance=13mm},	
		level 2/.style={sibling distance=7mm},
		level 3/.style={sibling distance=4mm},
		scale=1.5
		]
		\node[black, draw, circle, inner sep = 1.5 pt, fill=black, double]{}
		child{
			node[nb]{}
			child{
				node[nc] {}}
			child{
				node[nc] {}}	
		}
		child{
			node[nb]{}
			child{
				node[nb]{}
				child{
					node[nb]{}
					child {
						node[nc] {}}
					child{
						node[nb]{}
						child{
							node[nb]{}
							child{
								node[nc] {}
							}
							child{
								node[nc] {}
							}
						}
						child {
							node[nc] {}}
					}
				}
				child {
					node[nc] {} }
			}
			child{
				node[nb]{}
				child{
					node[nc] {}
				}
				child{
					node[nc] {}
				}
			}
		};
		
		\end{tikzpicture}
	\end{center}
	\caption{The left graph (colored in blue) and the right graph (colored in red and dashed) of the triangulation of $\polygon{10}$ in Figure~\ref{fig_tree_and_triangulation}; and the corresponding full binary tree $\mathcal{C}_T$. Here, we fill the vertices $V(\mathcal{C}_T) \setminus V(\mathcal{B}_T)$ with purple.}\label{fig_atom_and_coatom_together}
\end{figure}

From the definition of the left and the right graphs, 
we obtain the following lemma which proves that the left and the right graphs are indeed \emph{trees}.  
\begin{lemma}\label{lemma_uniqueness_kl_k_kr}
	Let $T$ be a triangulation of $\polygon{n+2}$.
	For each $1 \leq k \leq n$, there is only one edge $\{k_{\ll}, k\}$ with $k_{\ll} < k$ in the left graph $\lT$ and similarly there is only one edge $\{k, k_{\rr}\}$ with $k < k_{\rr}$ in the right graph $\rT$. Moreover, $k_{\ll}, k, k_{\rr}$ are the vertices of a triangle in $T$. Indeed, both graphs $\lT$ and $\rT$ are trees.
\end{lemma}
\begin{proof}
For each $1\le k\le n$, there is a unique triangle in $T$ which has $k$ as the middle vertex. Let $k_{\ell}<k<k_{r}$ be the vertices of the triangle.  By definition, the edge $\{k_{\ell},k\}$ is in the left tree $\lT$ while the edge $\{k,k_r\}$ is in the right tree $\rT$.  Since there are exactly $n$ edges in $\lT$ (resp. $\rT$), $\{k_\ell,k\}$ (resp. $\{k,k_r\}$) for $1\le k\le n$ provide all the edges in $\lT$ (resp. $\rT$).  This implies that $k_\ell$ (resp. $k_r$) is the desired $k_{\ll}$ (resp. $k_{\rr}$). 

We claim that $\lT$ and $\rT$ are trees. First consider the left graph $\lT$. 
Since there is an edge $\{k_{\ll}, k\}$ with $k_{\ll} < k$ for each $1 \le k \le n$, any vertex $k$ is connected to the root vertex $0$, so the graph $\lT$ is connected. Moreover, the number of edges is $n$. Accordingly, the left graph $\lT$ is a tree. Similarly, the right graph $\rT$ is connected and there are $n$ edges, so $\rT$ is a tree. This proves the lemma.
\end{proof}

As is proved in Lemma~\ref{lemma_uniqueness_kl_k_kr}, both graphs $\lT$ and $\rT$ are trees. So we will call them the \emph{left tree} and the \emph{right tree}, respectively. 

\begin{remark}\label{rmk_lT_rT_noncrossing}
	The set of left trees of polygon triangulations provides a model for the Catalan numbers.
More precisely, as is explained in~\cite[{\#21 in \S2}]{Stanley_Catalan}, there is a bijective correspondence between the set of triangulations and the set of noncrossing increasing trees on the vertex set $\{0,1,\dots,n\}$. The latter set is a set of trees whose vertices are arranged in increasing order around a circle such that no edges cross in their interior, and such that all paths from the root vertex $0$ are increasing. 
	On the other hand, the set of right trees is the set of noncrossing decreasing trees on the vertex set $\{1,2,\dots,n+1\}$ with the root vertex $n+1$.
	Indeed, we have 
	\[
	\#\{ \lT \mid T \text{ a triangulation of $\polygon{n+2}$}\} = \#\{ \rT \mid T \text{ a triangulation of $\polygon{n+2}$}\}= C_n.
	\]
\end{remark}

For each $k=1,\dots,n$, we define 
\begin{enumerate}
	\item $\mathbf{p}_k=e_{k_{\ll}+1}-e_{k+1}$,
	\item $\mathbf{q}_k=-e_k+e_{k_{\rr}}$,
\end{enumerate}
where $\{e_1,\dots,e_{n+1}\}$ is the standard basis of $\Z^{n+1}$. The vectors $\mathbf{p}_1,\dots, \mathbf{p}_n$ (similarly, $\mathbf{q}_1,\dots,\mathbf{q}_n$) form a basis of the sublattice $M$ of $\Z^{n+1}$ where
\[
M=\{(x_1,\dots,x_{n+1})\in \Z^{n+1}\mid x_1+\dots+x_{n+1}=0\}. 
\]

Through the dot product on $\Z^{n+1}$, the dual lattice $N$ of~$M$ can be identified with the quotient lattice of $\Z^{n+1}$ by the sublattice generated by $(1,\dots,1)$, i.e. 
\[
N=\Z^{n+1}/\Z(1,\dots,1).
\]
Let $\varpi_i$ $(i=0,1,\dots,n+1)$ be the quotient image of $\sum_{k=1}^i e_k$ in $N$. Then $\{\varpi_1,\dots,\varpi_n\}$ is a basis of $N$ and $\varpi_0=\varpi_{n+1}=\boldsymbol{0}$ by definition.

To each edge $\{a,b\}\in T$ with $a <b$, we assign the vector $\varpi_a-\varpi_b$ and denote it by $\mathbf{v}_a$ when $\{a,b\} \in \rT$ and $\mathbf{w}_b$ when $\{a,b\} \in \lT$, in other words, 
\begin{equation}\label{eq:defvkwk}
	\begin{split}
	\mathbf{v}_{k} &= \varpi_{k}- \varpi_{k_{\rr}},\\ 
	\mathbf{w}_k &= \varpi_{k_{\ll}} - \varpi_{k}, 
	\end{split}
\end{equation}
for $k=1,\dots,n$ by Lemma~\ref{lemma_uniqueness_kl_k_kr}.
Note that the zero vector $\mathbf{0}$ is assigned to the distinguished edge ${\{0,n+1\}}$ because $\varpi_0=\varpi_{n+1}=\mathbf{0}$.

With this understood, we have the following. 

\begin{proposition}\label{prop_dual}
	Let $\langle\ ,\ \rangle$ be the pairing between $N$ and $M$ induced from the dot product on~$\Z^{n+1}$. Then we have
	\[
	\langle\mathbf{v}_i,\mathbf{p}_j\rangle=\langle \mathbf{w}_i,\mathbf{q}_j\rangle=\delta_{ij},
	\]
	where $\delta_{ij}$ denotes the Kronecker delta. 
\end{proposition}

\begin{proof} 
We prove the statement by providing $\langle\mathbf{v}_i,\mathbf{p}_j\rangle=\delta_{ij}$
and $\langle \mathbf{w}_i,\mathbf{q}_j\rangle=\delta_{ij}$. We first consider the pairing $\langle\mathbf{v}_i,\mathbf{p}_j\rangle$. 
	By definition of $\mathbf{v}_k$, $\mathbf{p}_k$ and $\varpi_i$, we have 
	\begin{equation} \label{eq:vipj}
		\langle \mathbf{v}_i,\mathbf{p}_j\rangle=\langle \varpi_i-\varpi_{i_{\rr}},e_{j_{\ll}+1}-e_{j+1}\rangle=\langle -(e_{i+1}+\cdots+e_{i_{\rr}}),e_{j_{\ll}+1}-e_{j+1}\rangle.
	\end{equation}
	On the other hand, we see from the construction of the left and right trees that 
	\[
	[i+1,i_{\rr}]\cap \{j_{\ll}+1,j+1\}=\begin{cases} \{j+1\}\quad &\text{when $i=j$},\\
	\emptyset \text{ or } \{j_{\ll}+1,j+1\}\quad &\text{when $i\not=j$}.\end{cases}
	\]
	This together with \eqref{eq:vipj} implies $\langle \mathbf{v}_i,\mathbf{p}_j\rangle =\delta_{ij}$. 
	The proof of the latter identity $\langle \mathbf{w}_i,\mathbf{q}_j\rangle=\delta_{ij}$ is similar to the above. Indeed, by definition of $\mathbf{w}_k$, $\mathbf{q}_k$ and $\varpi_i$, we have 
	\begin{equation} \label{eq:wiqj}
		\langle \mathbf{w}_i,\mathbf{q}_j\rangle=\langle \varpi_{i_{\ll}}-\varpi_{i},-e_{j}+e_{j_{\rr}}\rangle=\langle -(e_{i_{\ll}+1}+\cdots+e_{i}),-e_j+e_{j_{\rr}}\rangle.
	\end{equation}
	On the other hand, we see from the construction of the left and right trees that 
	\[
	[i_{\ll}+1,i]\cap \{j_{\rr},j\}=\begin{cases} \{j\}\quad &\text{when $i=j$},\\
	\emptyset \text{ or } \{j_{\rr},j\}\quad &\text{when $i\not=j$}.\end{cases}
	\]
	This together with \eqref{eq:wiqj} implies $\langle \mathbf{w}_i,\mathbf{q}_j\rangle =\delta_{ij}$. Hence the result follows.
\end{proof}
\begin{example}\label{example_vk_wk}
	In the case of Example~\ref{example_LR_tree_from_triangulation}, the vectors $\mathbf{p}_k, \mathbf{q}_k, \mathbf{v}_k, \mathbf{w}_k$ for $1\le k\le 8$ are given as follows: 
	\begin{center}
		\begin{tabular}{c|c|ll|ll}
			\toprule
			$k$ & $k_{\ll} < k < k_{\rr}$ & $\mathbf{p}_k$ & $\mathbf{q}_k$ & $\mathbf{v}_k$ & $\mathbf{w}_k$ \\
			\midrule 
			$1$ & $0 < 1 < 2$ & $e_1-e_2$ & $-e_1+e_2$ & $\varpi_1 - \varpi_2$ & $\varpi_0 - \varpi_1 = -\varpi_1$ \\
			$2$ & $0 < 2 < 9$ & $e_1-e_3$ & $-e_2+e_9$ & $\varpi_2 - \varpi_9 = \varpi_2$ & $\varpi_0-\varpi_2 = -\varpi_2$ \\
			$3$ & $2 < 3 < 6$ & $e_3-e_4$ & $-e_3+e_6$ & $\varpi_3 - \varpi_6$ & $\varpi_2 - \varpi_3$ \\
			$4$ & $3 < 4 < 5$ & $e_4-e_5$ & $-e_4+e_5$ & $\varpi_4 - \varpi_5$ & $\varpi_3 - \varpi_4$ \\
			$5$ &$3 < 5 < 6$ & $e_4-e_6$ & $-e_5+e_6$ & $\varpi_5 - \varpi_6 $ & $\varpi_3 - \varpi_5$ \\
			$6$ & $2 < 6 < 7$ & $e_3-e_7$ & $-e_6+e_7$ & $\varpi_6 - \varpi_7$ & $\varpi_2 - \varpi_6$ \\
			$7$ & $2 < 7 < 9$ & $e_3-e_8$ & $-e_7+e_9$ & $\varpi_7 - \varpi_9=\varpi_7$ & $\varpi_2 - \varpi_7$\\
			$8$ & $7 < 8 < 9$ & $e_8-e_9$ & $-e_8+e_9$ & $\varpi_8 - \varpi_9 = \varpi_8$ & $\varpi_7 - \varpi_8$ \\
			\bottomrule
		\end{tabular}
	\end{center}
	See Figure~\ref{fig_vectors_left_right_trees} for the assignment of the vectors $\mathbf{w}_k$'s and $\mathbf{v}_k$'s to the left and right trees. 
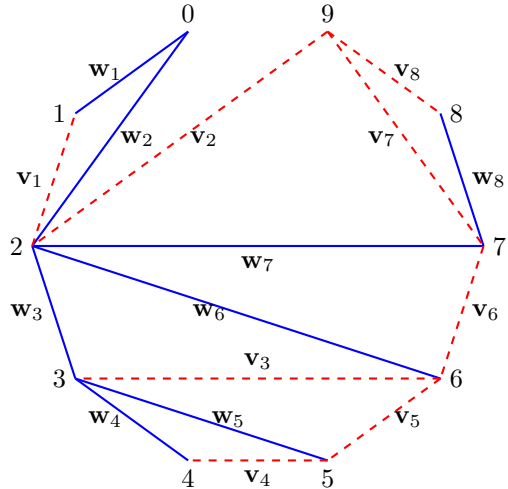
\begin{figure}[hbt]
\begin{center}
	\begin{tikzpicture}[x=10mm, y=5mm, label distance = 1mm, scale = 1]
	
\foreach \x in {1,2,...,11}{
	\coordinate (\x)  at (3*36-\x*36:3cm) ;	
} 

\node[above] at (11) {$9$};
\foreach \x/\y in {2/8,3/7,4/6}{
	\node[right] at (\x) {$\y$};
}
\foreach \x/\y in {5/5,6/4}{
	\node[below] at (\x) {$\y$};
}
\foreach \x/\y in {7/3,8/2,9/1}{
	\node[left] at (\x) {$\y$};
}
\node[above] at (10) {$0$};
	
\draw[blue, thick] (10)--(9) node[left, midway, black]{$\mathbf{w}_1$};
\draw[blue, thick] (10)--(8) node[right, midway, black]{$\mathbf{w}_2$};
\draw[blue, thick] (8)--(7) node[left, midway, black]{$\mathbf{w}_3$};
\draw[blue, thick] (8)--(4) node[left, midway, black]{$\mathbf{w}_6$};
\draw[blue, thick] (8)--(3) node[below, midway, black]{$\mathbf{w}_7$};
\draw[blue, thick] (7)--(6) node[left, midway, black]{$\mathbf{w}_4$};
\draw[blue, thick] (7)--(5) node[right, midway, black]{$\mathbf{w}_5$};
\draw[blue, thick] (3)--(2) node[right, midway, black]{$\mathbf{w}_8$};
\draw[red, dashed, thick] (1)--(2) node[right, midway, black]{$\mathbf{v}_8$};
\draw[red, dashed, thick] (1)--(3) node[left, midway, black]{$\mathbf{v}_7$};
\draw[red, dashed, thick] (1)--(8) node[right, midway, black]{$\mathbf{v}_2$};
\draw[red, dashed, thick] (3)--(4) node[right, midway, black]{$\mathbf{v}_6$};
\draw[red, dashed, thick] (5)--(6) node[below, midway, black]{$\mathbf{v}_4$};
\draw[red, dashed, thick] (8)--(9) node[left, midway, black]{$\mathbf{v}_1$};
\draw[red, dashed, thick] (4)--(5) node[right, midway, black]{$\mathbf{v}_5$};
\draw[red, dashed, thick] (4)--(7) node[above, midway, black]{$\mathbf{v}_3$};
	 
		\end{tikzpicture}
\end{center}
	\caption{Vectors assigned to the left and right trees of the triangulation $T$ of~$\polygon{10}$ in Figure~\ref{fig_tree_and_triangulation}}\label{fig_vectors_left_right_trees}
\end{figure}
\end{example}

\section{Left and right trees and binary trees}\label{sec_lT_rT_and_binary_trees}
Let $T$ be a triangulation of $\polygon{n+2}$ and let $\mathbf{v}_k$ and $\mathbf{w}_k$ $(k\in [n]\coloneqq\{1,\dots,n\})$ be the vectors associated with~$T$ defined in the previous section. 
In this section, we observe how the sum $\mathbf{v}_k + \mathbf{w}_k$ $(k\in [n])$ is  related to the binary tree $\B_T$ (see Corollary~\ref{coro:binary_trees}).
\begin{lemma} \label{lemm:fano_condition}
	There is a unique integer $k_0\in [n]$ such that $\mathbf{v}_{k_0} + \mathbf{w}_{k_0} = \mathbf{0}$ and for $k\in [n]\backslash\{k_0\}$ we have 
	\begin{equation} \label{eq:vkwksum}
		\mathbf{v}_k + \mathbf{w}_k = \begin{cases} \mathbf{v}_{k_{\ll}} & \text{ if } \{ k_{\ll}, k_{\rr}\} \in E(\rT), \\
			\mathbf{w}_{k_{\rr}} & \text{ if } \{k_{\ll}, k_{\rr}\} \in E(\lT).
		\end{cases}
	\end{equation} 
In fact, $k_0$ is the remaining vertex of the triangle containing the distinguished 
edge $\{0, n+1\}$.
\end{lemma}

\begin{proof}
The vertices $k_{\ll},k,k_{\rr}$ form a triangle in $T$ by Lemma~\ref{lemma_uniqueness_kl_k_kr} and we have 
\[
\mathbf{v}_k+\mathbf{w}_k=\varpi_{k_\ll}-\varpi_{k_\rr},
\]
which follows from~\eqref{eq:defvkwk}. When $\{k_\ll,k_\rr\}=\{0,n+1\}$, we have $\varpi_{k_\ll}-\varpi_{k_\rr}=\boldsymbol{0}$ because \linebreak $\varpi_0=\varpi_{n+1}=\boldsymbol{0}$ by definition.
Otherwise, $\{k_\ll,k_\rr\}$ is an edge of $\lT$ or $\rT$. 
By Lemma~\ref{lemma_uniqueness_kl_k_kr}, we obtain $(k_{\ll})_{\rr} = k_{\rr}$ if $\{k_{\ll}, k_{\rr}\} \in E(\rT)$; otherwise, we obtain $(k_{\rr})_{\ll}=k_{\ll}$.
Hence, again by~\eqref{eq:defvkwk}, we get
\begin{equation*}
		\varpi_{k_\ll}-\varpi_{k_\rr} = \begin{cases} \mathbf{v}_{k_{\ll}} & \text{ if } \{ k_{\ll}, k_{\rr}\} \in E(\rT), \\
			\mathbf{w}_{k_{\rr}} & \text{ if } \{k_{\ll}, k_{\rr}\} \in E(\lT).
		\end{cases}
	\end{equation*} 
	This proves the lemma.
\end{proof}

\begin{example}\label{example_primitive_relation}
	The direct computation of the sum $\mathbf{v}_k + \mathbf{w}_k$ of the vectors $\mathbf{v}_k$, $\mathbf{w}_k$ provided in Example~\ref{example_vk_wk} shows the following: 
	\begin{align*}
		\mathbf{v}_1 + \mathbf{w}_1 &= (\varpi_1 - \varpi_2) + (-\varpi_1) = \mathbf{w}_2, &
		\mathbf{v}_2 + \mathbf{w}_2 &= \varpi_2 + (-\varpi_2) = \mathbf{0}, \\
		\mathbf{v}_3 + \mathbf{w}_3 &= (\varpi_3 - \varpi_6) + (\varpi_2 - \varpi_3) = \mathbf{w}_6, &
		\mathbf{v}_4 + \mathbf{w}_4 &= (\varpi_4 - \varpi_5) + (\varpi_3 - \varpi_4) = \mathbf{w}_5, \\
		\mathbf{v}_5 + \mathbf{w}_5 &= (\varpi_5 - \varpi_6) + (\varpi_3 - \varpi_5) = \mathbf{v}_3, &
		\mathbf{v}_6 + \mathbf{w}_6 &= (\varpi_6 -\varpi_7) + (\varpi_2 - \varpi_6) = \mathbf{w}_7, \\
		\mathbf{v}_7+\mathbf{w}_7 &= \varpi_7 + (\varpi_2 - \varpi_7) = \mathbf{v}_2, &
		\mathbf{v}_8 + \mathbf{w}_8 &= \varpi_8 + (\varpi_7 - \varpi_8) = \mathbf{v}_7.
	\end{align*}
	One can check that Lemma~\ref{lemm:fano_condition} holds from this computation. For example, when $k = 1$, for the triangle with vertices $0 < 1 < 2$, the edge $\{0,2\}$ is in the left tree, so the lemma says that $\mathbf{v}_1 + \mathbf{w}_1= \mathbf{w}_2$, which agrees with the computation above. Each relation above corresponds to a triangle in Figure~\ref{fig_vectors_left_right_trees}.
\end{example}

Motivated by Lemma~\ref{lemm:fano_condition}, we define a function $\varphi =\varphi_T \colon [n]\setminus \{k_0\}\to [n]$ by 
\begin{equation}\label{eq_def_varphi}
\varphi(k)=\begin{cases}
{k_{\ll}} & \text{ if } \{ k_{\ll}, k_{\rr}\} \in E(\rT), \\
{k_{\rr}} & \text{ if } \{k_{\ll}, k_{\rr}\} \in E(\lT), 
\end{cases}
\end{equation}
and a \defi{sign map} $\sigma=\sigma_T \colon [n]\setminus \{k_0\} \to \{ +, -\}$ by 
\[
\sigma(k) = \begin{cases}
+ & \text{ if } \{ k_{\ll}, k_{\rr}\} \in E(\rT), \\
- & \text{ if } \{k_{\ll}, k_{\rr}\} \in E(\lT). 
\end{cases}
\]

\begin{proposition} \label{lemm:binary}
	We have $|\varphi^{-1}(\ell)|\le 2$ for any $\ell\in [n]$, and if $\varphi(k_1)=\varphi(k_2)$ for $k_1\not=k_2$, then $\sigma(k_1)\not=\sigma(k_2)$. 
\end{proposition}

\begin{proof}
	The vertices $k_\ll,k,k_\rr$ form a triangle in $T$ for $k=1,\dots,n$ by   Lemma~\ref{lemma_uniqueness_kl_k_kr} and we often denote the triangle by $k_\ll<k<k_\rr$ in the following. 
	Let $k_1 < k_2$ and $\varphi(k_1) = \varphi(k_2)=\ell$. 
	We claim that the triangles $(k_1)_{\ll} < k_1 < (k_1)_{\rr}$ and $(k_2)_{\ll} < k_2 < (k_2)_{\rr}$ must be adjacent to the triangle $\ell_{\ll}<\ell<\ell_{\rr}$ as depicted in Figure~\ref{fig_triangle_kpq3}, 
which implies the former assertion $|\varphi^{-1}(\ell)|\le 2$ in the proposition.

We consider the following four cases according as $\varphi(k_i)=(k_i)_{\ll}$ or $(k_i)_{\rr}$ for $i=1,2$ by~\eqref{eq_def_varphi}:
	\begin{enumerate}
		\item $(k_1)_{\ll} =\varphi(k_1)=\ell=\varphi(k_2)= (k_2)_{\ll}$;
		\item $(k_1)_{\rr} =\varphi(k_1)=\ell=\varphi(k_2)= (k_2)_{\rr}$; 
		\item $(k_1)_{\ll} =\varphi(k_1)=\ell=\varphi(k_2)= (k_2)_{\rr}$; or
		\item $(k_1)_{\rr} =\varphi(k_1)=\ell=\varphi(k_2)= (k_2)_{\ll}$.
	\end{enumerate}
	Since $k_{\ll} < k < k_{\rr}$  for any $k$  and $k_1 < k_2$,  we obtain $(k_1)_{\ll} < k_1 < k_2 < (k_2)_{\rr}$, so case (3) does not occur. 
	For case (1),  the triangles $(k_1)_{\ll} < k_1 < (k_1)_{\rr}$ and $(k_2)_{\ll} < k_2 < (k_2)_{\rr}$ share the vertex $\ell$ and
	\[
	\{\ell, (k_1)_{\rr}\}, \{ \ell, (k_2)_{\rr}\} \in E(\rT)
	\]
	by \eqref{eq_def_varphi} as depicted in Figure~\ref{fig_triangle_kpq1}. However, this contradicts Lemma~\ref{lemma_uniqueness_kl_k_kr} since $\ell < (k_i)_{\rr}$ for $i = 1,2$ and $(k_1)_{\rr}\not=(k_2)_{\rr}$, so case (1) does not occur. Similarly case (2) also does not occur, see Figure~\ref{fig_triangle_kpq2}.  Finally, for case (4), since $\{(k_1)_{\ll},(k_1)_{\rr}\}$ is an edge in $\lT$ and $(k_1)_{\ll}<(k_1)_{\rr}=\ell$, we have $(k_1)_{\ll}=\ell_{\ll}$ by Lemma~\ref{lemma_uniqueness_kl_k_kr}.  Similarly we have $(k_2)_{\rr}=\ell_{\rr}$, see Figure~\ref{fig_triangle_kpq3}.  Therefore, the triangles $(k_1)_{\ll} < k_1 < (k_1)_{\rr}$ and $(k_2)_{\ll} < k_2 < (k_2)_{\rr}$ are adjacent to the triangle $\ell_{\ll}<\ell<\ell_{\rr}$ as depicted in Figure~\ref{fig_triangle_kpq3}. This proves the claim. 

Moreover, in case (4), 
	\[
	\varphi(k_1) =\ell= (k_1)_{\rr},\ \text{so}\ \sigma(k_1) = -, \qquad \text{ and }\qquad 
	\varphi(k_2) =\ell= (k_2)_{\ll},\ \text{so}\ \sigma(k_2) = +. 
	\] 
This shows the latter assertion in the proposition.
\end{proof}
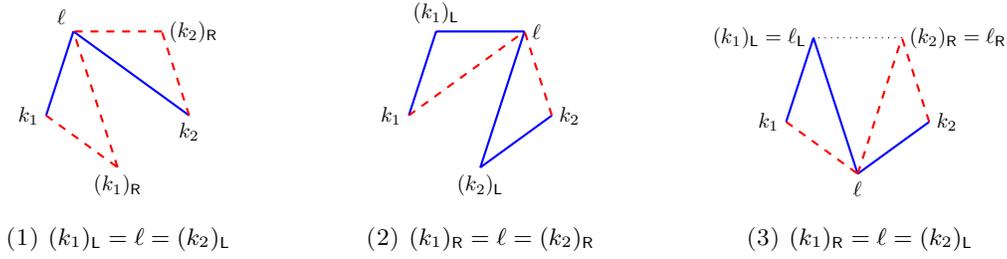
\begin{figure}[H]
	\begin{subfigure}[b]{0.3\textwidth}
		\begin{center}
	\begin{tikzpicture}[x=10mm, y=5mm, label distance = 1mm, every node/.style={scale=0.8}]
	\foreach \x in {1,2,3,4,5}{
	\coordinate (a\x) at (54+\x*72:1cm) ; 
	} 
	\draw[blue, thick] (a1)--(a2)
	(a1)--(a4);
	
	\draw[red, dashed, thick] 
	(a2)--(a3)
	(a4)--(a5)
	(a3)--(a1)
	(a5)--(a1);

\node[above left] at (a1) {$\ell$};
\node[left] at (a2) {$k_1$};
\node[below] at (a4) {$k_2$};
\node[right] at (a5) {$(k_2)_{\rr}$};
\node[below] at (a3) {$(k_1)_{\rr}$};

	\end{tikzpicture}
\end{center}
	
\caption{$(k_1)_{\ll} = \ell =  (k_2)_{\ll}$}\label{fig_triangle_kpq1}
	\end{subfigure}
	\begin{subfigure}[b]{0.33\textwidth}
	\begin{center}
		\begin{tikzpicture}[x=10mm, y=5mm, label distance = 1mm, every node/.style={scale=0.8}]
	\foreach \x in {1,2,3,4,5}{
	\coordinate (a\x) at (54+\x*72:1cm) ; 
} 
			\draw[blue, thick] (a1)--(a2)
			(a3)--(a4)
			(a5)--(a1)
			(a3)--(a5);
			
			\draw[red, dashed, thick] (a2)--(a5)
			(a4)--(a5);

			\node[above] at (a1) {$(k_1)_{\ll}$};
			\node[left] at (a2) {$k_1$};
			\node[below] at (a3) {$(k_2)_{\ll}$};
			\node[right] at (a4) {$k_2$};
			\node[right] at (a5) {$\ell$};
			
		\end{tikzpicture}
	\end{center}
	
	\caption{$(k_1)_{\rr} = \ell= (k_2)_{\rr}$}\label{fig_triangle_kpq2}
\end{subfigure}
	 \begin{subfigure}[b]{0.33\textwidth}
	\begin{center}
		\begin{tikzpicture}[x=10mm, y=5mm, label distance = 1mm, every node/.style={scale=0.8}]
			\foreach \x in {2,...,5,6}{
				\coordinate (a\x) at (-18-\x*72:1cm) ; 
			} 
			
			\draw[blue, thick] (a3)--(a2)
			(a3)--(a6)
			(a6)--(a5);
			
			\draw[red, dashed, thick] (a2)--(a6)
			(a6)--(a4)
			(a5)--(a4); 
			
			\node[right] at (a4) {{$(k_2)_{\rr}=\ell_{\rr}$}};
			\node[right] at (a5) {$k_2$};
			\node[below] at (a6) {{$\ell$}};
			\node[left] at (a2) {$k_1$};
			\node[left] at (a3) {{$(k_1)_{\ll}=\ell_{\ll}$}};
			
			\draw[dotted] (a4)--(a3);
		\end{tikzpicture}
	\end{center}
	\caption{$(k_1)_{\rr} = {\ell} = (k_2)_{\ll}$}\label{fig_triangle_kpq3}
\end{subfigure}	

\caption{Triangles in the proof of Proposition~\ref{lemm:binary}}
\end{figure}\label{fig_triangle_proof2}

\begin{definition}\label{def_binary_tree_2}
	We define the (vertex-labeled) binary tree $\B_{(\varphi,\sigma)}$ associated with $\varphi$ and $\sigma$ as follows: 
	\begin{itemize}
		\item $V(\B_{(\varphi,\sigma)}) = [n]$ and the root vertex is the element $k_0$ defined in Lemma~\ref{lemm:fano_condition}.
		\item For each $k \in [n] \setminus \{k_0\}$, its parent is $\varphi(k)$. Moreover, if $\sigma(k) = -$, then the vertex $k$ is the left child of $\varphi(k)$ and if $\sigma(k)=+$, then the vertex $k$ is the right child of $\varphi(k)$.
	\end{itemize}
\end{definition}

For each $1 \le k \le n$, there is a unique triangle in $T$ having $k$ as the middle vertex by Lemma~\ref{lemma_uniqueness_kl_k_kr}. This provides the vertex-labeling on the binary tree $\B_{T}$ defined in Section~\ref{sec_Catalan_numbers}. 
The proof of Proposition~\ref{lemm:binary} implies the following. 

\begin{corollary}\label{coro:binary_trees}
Let $T$ be a triangulation of $\polygon{n+2}$ and  $\B_T$  the binary tree 
defined in Section~\ref{sec_Catalan_numbers}.
Considering the vertex-labeling on~$\B_T$ given by 
Lemma~\ref{lemma_uniqueness_kl_k_kr}, the binary tree $\B_T$ is the same as 
$\B_{(\varphi,\sigma)}$ defined above as a vertex-labeled tree. 
\end{corollary}
Figure~\ref{fig:vertex-labeled tree} provides an example of the binary tree $\B_{(\varphi, \sigma)}$. Comparing this binary tree $B_{(\varphi, \sigma)}$ with the one in Figure~\ref{fig_tree_and_triangulation}, we obtain an example of Corollary~\ref{coro:binary_trees}.
	\begin{figure}[hbt]
			\centering
			\begin{tikzpicture}[vertex/.style = {circle, draw, inner sep = 1pt },
			level distance=5mm,	
			level 1/.style={sibling distance=13mm},	
			level 2/.style={sibling distance=7mm},
			level 3/.style={sibling distance=5mm},
			scale=2,
			]
			\node[vertex]{$2$}
			child{
				node[vertex]{$1$}
				edge from parent node[above left] {$-$}
			}
			child{
				node[vertex]{$7$} 
				child{
					node[vertex]{$6$}
					child{
						node[vertex]{$3$}
						child[missing] {}
						child{
							node[vertex]{$5$}
							child{
								node[vertex]{$4$} 
								edge from parent node[above left] {$-$}
							}
						child[missing] {}
						edge from parent node[above right] {$+$}
						}
					edge from parent node[above left] {$-$}
					}
					child[missing] {}
				edge from parent node[above left] {$-$}	
				}
				child{
					node[vertex]{$8$} edge from parent node[above right] {$+$}
				}
			edge from parent node[above right] {$+$}
			};
			
			\end{tikzpicture}
	\caption{The vertex-labeled binary tree $\B_{(\varphi, \sigma)}$ associated  
	with the triangulation~$T$ of~$\polygon{10}$ in 
	Figure~\ref{fig_tree_and_triangulation}.}\label{fig:vertex-labeled tree}
\end{figure}
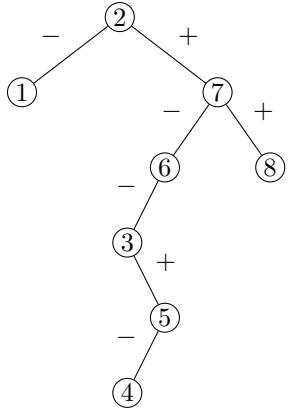
\section{Fans and toric varieties of Catalan type} \label{sec:fan_Catalan_type}
Let $T$ be a triangulation of $\polygon{n+2}$ and let $\mathbf{v}_k$ and $\mathbf{w}_k$ $(k\in [n])$ be the vectors associated with~$T$ defined in Section~\ref{sec_left_and_right_trees}. In this section, we provide fans and toric varieties \emph{of Catalan type} using the vectors $\mathbf{v}_k, \mathbf{w}_k$ and study when such fans are isomorphic.

Before describing the fan corresponding to a triangulation $T$, we prepare one terminology. 
Let $\Sigma$ be a fan. A subset $P$ of the primitive ray vectors in $\Sigma$ is called a \defi{primitive collection} of $\Sigma$ if 
\[
\Cone(P) \notin \Sigma
\quad \text{ but }\quad \Cone(P \setminus \{\mathbf{u}\}) \in \Sigma \qquad \text{ for every }\mathbf{u} \in P.
\] 
We denote by $\PC(\Sigma)$ the set of primitive collections of $\Sigma$.
\begin{example}\label{example_primitive_collection}
Suppose that $\Sigma$ consists of four ray vectors $\mathbf{u}_1=(1,0)$, $\mathbf{u}_2 = (0,1)$, $\mathbf{u}_3 = (-1,0)$, $\mathbf{u}_4 = (-1,-1)$; and four maximal cones $\Cone(\mathbf{u}_1,\mathbf{u}_2), \Cone(\mathbf{u}_1, \mathbf{u}_4), \Cone(\mathbf{u}_2,\mathbf{u}_3), \Cone(\mathbf{u}_3, \mathbf{u}_4)$ as shown in Figure~\ref{figure_fan_example}. 
Then the set of primitive collections of $\Sigma$ is $\{\{\mathbf{u}_1,\mathbf{u}_3\}, \{\mathbf{u}_2, \mathbf{u}_4\}\}$.
\end{example}
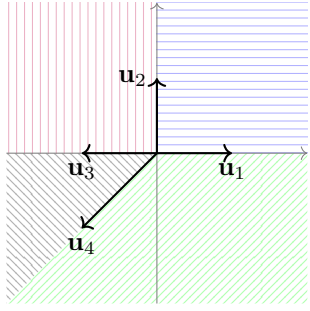
\begin{figure}
\begin{tikzpicture}
\fill[pattern color = blue!30!white, pattern = horizontal lines] (0,0)--(2,0)--(2,2)--(0,2)--cycle;
\fill[pattern color = purple!30!white, pattern = vertical lines] (0,0)--(0,2)--(-2,2)--(-2,0)--cycle;
\fill[pattern color= black!30, pattern = north west lines] (0,0)--(-2,0)--(-2,-2)--cycle;
\fill[pattern color=green!30, pattern = north east lines] (0,0)--(2,0)--(2,-2)--(-2,-2)--cycle;

\draw[->, gray] (-2,0)--(2,0);
\draw[->, gray] (0,-2)--(0,2);

\draw[thick, ->] (0,0)--(1,0) node[below] {$\mathbf{u}_1$};
\draw[thick, ->] (0,0)--(0,1) node[left] {$\mathbf{u}_2$};
\draw[thick, ->] (0,0)--(-1,0) node[below] {$\mathbf{u}_3$};
\draw[thick, ->] (0,0)--(-1,-1) node[below] {$\mathbf{u}_4$};

\end{tikzpicture}
\caption{A fan in Example~\ref{example_primitive_collection}. The set of primitive collections is $\{\{\mathbf{u}_1,\mathbf{u}_3\}, \{\mathbf{u}_2, \mathbf{u}_4\}\}$.}\label{figure_fan_example}
\end{figure}

\begin{lemma} \label{lemm:fan}
Let $T$ be a triangulation of $\polygon{n+2}$ and let $\mathbf{v}_k$ and $\mathbf{w}_k$ $(k\in [n])$ be the vectors associated with~$T$. 
Define a set $\Sigma_T$ of cones as follows: A subset $Q \subset \{ \mathbf{v}_k,\mathbf{w}_k \mid k \in [n]\}$ forms a cone $\Cone(Q) \in \Sigma_T$ if and only if $ \{\mathbf{v}_i, \mathbf{w}_i\} \not\subset Q$ for all $i \in [n]$.
Then $\Sigma_T$ is a complete non-singular fan in $N\otimes \R$. 
Moreover, the set $\PC(\Sigma_T)$ of primitive collections of the fan $\Sigma_T$ is $\PC(\Sigma_T) = \{\{\mathbf{v}_k, \mathbf{w}_k \} \mid k \in [n]\}$.
\end{lemma}

\begin{proof}
	Take a permutation $u\in \mathfrak{S}_n$ such that $u(k_0)=1$ and $u (\varphi(k))<u (k)$ for any $k\in [n]\setminus\{k_0\}$. 
	We notice that the existence of $u$ follows from the fact that $\B_{(\varphi, \sigma)}$ is a tree. 
	Then, it follows from Lemma~\ref{lemm:fano_condition} that the transition matrix $A$ defined by 
	\[
	(\mathbf{w}_{u^{-1}(1)},\dots,\mathbf{w}_{u^{-1}(n)})=(\mathbf{v}_{u^{-1}(1)},\dots,\mathbf{v}_{u^{-1}(n)})A
	 \]
	 is an upper triangular $n\times n$ integer matrix with $-1$ as diagonals. 
This implies that there is an automorphism of the lattice which sends the vectors $ \mathbf{v}_{u^{-1}(1)},\dots,\mathbf{v}_{u^{-1}(n)}$ to the standard basis vectors; and the vectors $\mathbf{w}_{u^{-1}(1)},\dots,\mathbf{w}_{u^{-1}(n)}$ to the column vectors of the matrix $A$. 
On the other hand, it is known from~\cite[Proposition~3.2]{MasudaPanov08} that the set $\Sigma_T$ of cones made by such vectors always form a complete non-singular fan. This proves the lemma.
\end{proof}

\begin{example}
	Suppose that $T$ is the triangulation in Figure~\ref{fig_tree_and_triangulation}. As we saw in Example~\ref{example_primitive_relation},  $k_0 = 2$ and
	\[
	\varphi(1) = 2, \quad \varphi(3) = 6,\quad \varphi(4)=5,\quad \varphi(5) = 3,\quad \varphi(6)=7,\quad \varphi(7)=2, \quad\varphi(8)=7,
	\]
	where $\varphi \colon[8] \setminus \{2\} \to [8]$ is defined in~\eqref{eq_def_varphi}. If we take $u = 21687534$, then it satisfies \linebreak $u(k_0) = u(2) = 1$ and $u (\varphi(k)) < u (k)$ for any $k \in [8] \setminus \{2\}$, and the transition matrix $A$ defined by $(\mathbf{w}_{u^{-1}(1)},\dots,\mathbf{w}_{u^{-1}(8)})=(\mathbf{v}_{u^{-1}(1)},\dots,\mathbf{v}_{u^{-1}(8)})A$ is
	\[
	A = \begin{pmatrix}
	-1& -1&  1&  0&  1&  1&  0&  0\\
	 0& -1&  0&  0&  0&  0&  0&  0\\
	 0&  0& -1&  1& -1& -1&  0&  0\\
	 0&  0&  0& -1&  0&  0&  0&  0\\
	 0&  0&  0&  0& -1& -1&  0&  0\\
	 0&  0&  0&  0&  0& -1&  1&  1\\
	 0&  0&  0&  0&  0&  0& -1& -1\\
	 0&  0&  0&  0&  0&  0&  0& -1
	\end{pmatrix},
	\]
	which is an upper triangular matrix with $-1$ as diagonals as we expected.
\end{example}

\begin{definition}
	For a polygon triangulation $T$, we say that the fan $\Sigma_T$ and the corresponding (smooth compact) toric variety $X(\Sigma_T)$ are of \emph{Catalan type}. 
\end{definition}

Lemma~\ref{lemm:fan} says that the underlying simplicial complex of the fan $\Sigma_T$ is the boundary complex of an $n$-dimensional cross-polytope. It is known from~\cite[Corollary~3.5]{MasudaPanov08} that such a fan is indeed the normal fan of an $n$-cube, so $X(\Sigma_T)$ is projective. 

Let $\Sigma$ be a complete non-singular fan and let $\PC(\Sigma)$ be the primitive collections. 
For a primitive collection $P = \{\mathbf{u}'_1, \dots,\mathbf{u}'_{\ell}\}$, there exists a unique cone $\sigma$ such that $\mathbf{u}'_1 + \cdots+\mathbf{u}'_{\ell}$ is in the interior of $\sigma$. Let $\mathbf{u}_1,\dots,\mathbf{u}_{r}$ be the primitive generators of $\sigma$. Then, there exist positive integers $a_1,\dots,a_{r}$ such that 
	\begin{equation}\label{equation_primitive_relation}
	\mathbf{u}'_1+\cdots+\mathbf{u}'_{\ell} = a_1 \mathbf{u}_1 + \cdots+ a_{r} \mathbf{u}_{r}.
	\end{equation}
(If the sum is the zero vector, then the set $\{\mathbf{u}_1,\dots,\mathbf{u}_{r}\}$ is assumed to be empty.) We call the linear relation in~\eqref{equation_primitive_relation} the \emph{primitive relation} corresponding to the primitive collection $P$; and we define the \emph{degree} $\deg P$ of the primitive collection $P$ as the number $\ell - (a_1+\cdots+a_r)$. Using Batyrev's criterion~\cite{Batyrev} for a smooth projective toric variety, we obtain the following. 
\begin{lemma} \label{lemm:Fano}
	The toric variety $X(\Sigma_T)$ is Fano. 
\end{lemma}
\begin{proof}
	Let $P$ be any primitive collection of $\Sigma_T$. Then $P=\{\mathbf{v}_k,\mathbf{w}_k\}$ for some $k\in [n]$ by Lemma~\ref{lemm:fan} and $\deg P>0$ by Lemma~\ref{lemm:fano_condition}, indeed $\deg P=2$ when $k=k_0$ and $1$ otherwise. Therefore, $X(\Sigma_T)$ is Fano by Batyrev's criterion~\cite[Proposition~2.3.6]{Batyrev}. 
\end{proof}

We say that two binary trees $\B$ and $\B'$ are isomorphic as \textit{unordered rooted trees} if there is a bijection $f \colon V(\B) \to V(\B')$ between the set of vertices which sends the root of $\B$ to the root of~$\B'$ and induces a bijection between the set of edges (which does not need to preserve the \textit{left} and \textit{right} children). We note that a binary tree obtained from $\B$ by interchanging the left and right children of a vertex is isomorphic to $\B$ and any isomorphism is obtained by a composition of this operation at vertices. 

\begin{theorem}\label{theo:fan_and_toric_of_Catalan}
	Let $T$ and $T'$ be triangulations of $\polygon{n+2}$. Then the fans $\Sigma_T$ and $\Sigma_{T'}$ are isomorphic \textup{(}equivalently, the toric varieties $X(\Sigma_T)$ and $X(\Sigma_{T'})$ are isomorphic\textup{)} if and only if the binary trees~$\B_{T}$ and $\B_{T'}$ are isomorphic as unordered rooted trees.
\end{theorem}

\begin{proof}
We first notice that both fans $\Sigma_T$ and $\Sigma_{T'}$ define Fano toric varieties by Lemma~\ref{lemm:Fano}. Recall from~\cite[Proposition~2.1.8, Theorem~2.2.4]{Batyrev} that two Fano toric varieties $X(\Sigma)$ and $X(\Sigma')$ for $\Sigma \subset N \otimes \R$ and $\Sigma' \subset N' \otimes \R$ are isomorphic if and only if there exists a lattice isomorphism $\Phi \colon N \to N'$ providing a bijection between  $\Sigma$ and $\Sigma'$ such that $\Phi$ preserves each primitive collection and their relation. 

We denote the lattice and vectors associated with $T'$ by putting a prime. If the fans~$\Sigma_T$ and~$\Sigma_{T'}$ are isomorphic, then there is an isomorphism $\Phi$ between the lattices $N$ and $N'$ preserving the cones and primitive collections in $\Sigma_T$ and $\Sigma_{T'}$. Therefore, for each $i \in [n]$, there exists an index $\eta(i)\in[n]$ such that $\{\Phi(\mathbf v_i), \Phi(\mathbf w_i)\} = \{\mathbf v_{\eta(i)}', \mathbf w_{\eta(i)}'\}$ holds, so $\Phi(\mathbf v_i)$ is either $\mathbf v_{\eta(j)}'$ or $\mathbf w_{\eta(j)}'$. Since $\Phi \colon N \to N'$ is an isomorphism, $\eta \colon [n] \to [n]$ is a bijection.
Note that for a triangulation $T$, the binary tree $\B_T$ encodes the information of primitive relations in $\Sigma_T$ by Corollary~\ref{coro:binary_trees}. Indeed, for each $k \in [n] \setminus \{k_0\}$, we have a primitive relation
\[
\mathbf{v}_k + \mathbf{w}_k = \begin{cases}
\mathbf{v}_{\varphi(k)} & \text{ if $k$ is the left child of $\varphi(k)$}, \\
\mathbf{w}_{\varphi(k)} & \text{ if $k$ is the right child of $\varphi(k)$}.
\end{cases}
\]
Since $\Phi$ preserves the primitive relation, the bijection $\eta$ provides an isomorphism between the binary trees $\B_{T}$ and $\B_{T'}$ as unordered rooted trees.

Suppose that two binary trees $\B_{T}$ and $\B_{T'}$ are isomorphic as unordered rooted trees. We already notice that any isomorphism between unordered rooted trees is obtained by a composition of interchanging the left and right children of a vertex. Accordingly, we have a lattice isomorphism preserving each primitive collection and their relation. Hence two fans $\Sigma_T$ and~$\Sigma_{T'}$ are isomorphic and the result follows. 
\end{proof}

By Theorem~\ref{theo:fan_and_toric_of_Catalan}, the number of isomorphism classes of $n$-dimensional toric varieties of Catalan type is the same as that of unordered binary trees with $n$ vertices. It is known that the latter is the Wedderburn--Etherington number $b_{n+1}$ mentioned in the introduction.

\begin{corollary}\label{cor_enumerate}
	The number of isomorphism classes of $n$-dimensional toric varieties of Catalan type is the Wedderburn--Etherington number $b_{n+1}$. 
\end{corollary} 

\begin{remark}\label{rmk_Bott_manifolds}
A smooth projective toric variety associated with the normal fan of a 
non-singular lattice $n$-cube is called a \textit{Bott manifold}.
A Bott manifold can also be obtained as the total space of an iterated $\C 
P^1$-bundle over a point called a \emph{Bott tower} (see~\cite{GK94Bott}). So, 
our toric variety $X(\Sigma_T)$ is a 
Bott manifold and Fano. A Fano Bott manifold is not necessarily of Catalan type, 
i.e. not necessarily of the form~$X(\Sigma_T)$. However, one can associate a 
graph, called a \textit{signed rooted forest}, with a Fano Bott manifold, and Fano 
Bott manifolds can be classified in terms of signed rooted forests 
(see~\cite{HigashitaniKurimoto20, CLMP}). As one may expect, the signed rooted 
forest associated with $X(\Sigma_T)$ is the binary tree $\B_T = 
\B_{(\varphi,\sigma)}$ with the sign~$\sigma(k)$ assigned to each edge 
$\{k,\varphi(k)\}$, see Figure~\ref{fig:vertex-labeled tree}.
\end{remark}

\section{Permutations, polygon triangulations, and Bruhat interval polytopes} \label{sec:permutations_etc}

As explained in Section~\ref{sec_Catalan_numbers}, there is a canonical bijection between the set of binary trees with~$n$ vertices and that of triangulations of $\polygon{n+2}$. In this section, we will explain how a binary tree having $n$ vertices, equivalently a triangulation of $\polygon{n+2}$, is associated with a permutation in~$\mathfrak{S}_n$.
The permutation is also associated with a Bruhat interval polytope, which is combinatorially equivalent to an $n$-cube in our case. The vectors $\mathbf{v}_1,\dots,\mathbf{v}_n, \mathbf{w}_1,\dots,\mathbf{w}_n$ introduced in Section~\ref{sec_left_and_right_trees} turn out to be the facet normal vectors of the Bruhat interval polytope.

We recall a surjection from the set $\mathfrak{S}_n$ of permutations on $[n]$ to that of binary trees with $n$ vertices (cf.~\cite[Appendix~A]{LFCG_dialgebras}). 
{To} a permutation $u \in \mathfrak{S}_n$, we associate a binary tree $\psi(u)$ by finding the smallest number in the one-line notation of $u$ inductively. We start with the one-line notation $u(1) u(2) \ \cdots \ u(n)$ of $u$. 
The smallest integer, say $u(p)$, in the sequence (which is $1$ here) becomes the root of the binary tree $\psi(u)$ with $n$ vertices. Then the subsequence $u(1) \ \cdots \ u(p-1)$ will provide the left subtree of the root vertex, and the subsequence $u(p+1) \ \cdots \ u(n)$ will provide the right subtree of the root vertex. 
More precisely, the smallest integer in the sequence $u(1) \ \cdots \ u(p-1)$ presents the root of a binary tree with $p-1$ vertices, and it is the left child of the root vertex of~$\psi(u)$. On the other hand, the smallest integer in the sequence $u(p+1) \ \cdots \ u(n)$ presents the root of a binary tree with $n-p$ vertices, and it is the right child of the root vertex of $\psi(u)$. Continuing this process, we get the binary tree $\psi(u)$. 

The binary tree $\psi(u)$ can also be obtained by drawing vertices on the $n\times n$ grid. For a permutation $u \in \mathfrak{S}_n$, we put $n$ vertices on $(u(i),i)$ position for $i=1,\dots,n$. (Here, we read the coordinates from top to bottom and left to right, like matrix coordinates.) The vertex~$(1, u^{-1}(1))$ is placed in the first row, and it presents the root. We measure a distance between two vertices by taking the difference in row position. By connecting the left (resp. right) closest point of the root vertex, we obtain the left (resp. right) child. Continuing this process, we obtain the binary tree $\psi(u)$. 

For example, let $u = 31687524$. The binary tree $\psi(u)$ has the root at $(1,2)$ and the root has two children such that the left subtree is given by $u(1) = 3$, which is the tree consisting of one vertex, that is, the left child of the root is $(3,1)$.
The right subtree is given by $u(3) \ u(4) \ \cdots \ u(8) = 687524$, which provides a binary tree with $6$ vertices and the number $2$ will present the root vertex, i.e., the right child of the vertex $(1,2)$ is $(2,7)$. We depict the $8 \times 8$ grid and the corresponding tree~$\psi(31687524)$ in Figure~\ref{fig_binary_tree_and_permutation}; the binary trees $\psi(u)$ of permutations $u \in \mathfrak{S}_3$ {in} Figure~\ref{fig_permutation_and_tree}.
\begin{figure}[h]
	\begin{center}\begin{tikzpicture}[scale=0.6,dot/.style = {circle, fill=black,  inner sep = 1.5pt, },
		level distance=5mm,	
		level 1/.style={sibling distance=5mm},	
		level 2/.style={sibling distance=5mm},
		level 3/.style={sibling distance=5mm}	
		]
		
		\foreach\l[count=\y] in {8,...,1}{
			\draw[gray!30](1,\y)--(8,\y);
			\draw[gray!30] (\y,1)--(\y,8);
			\node[gray] at (0.5,\y) {\small$\l$};
			\node[gray] at (\y,8.5) {\small$\y$};
			
		}
		
		\node[dot] (1) at (1,6) {};
		\node[black, draw, circle, inner sep = 1.5 pt, fill=black, double] (2) at (2,8) {};
		\node[dot] (3) at (3,3) {};
		\node[dot] (4) at (4,1) {};
		\node[dot] (5) at (5,2) {};
		\node[dot] (6) at (6,4) {};
		\node[dot] (7) at (7,7) {};
		\node[dot] (8) at (8,5) {};
		
		\draw (1)--(2) 
		(2)--(7) 
		(7)--(8) 
		(7)--(6) 
		(6)--(3) 
		(3)--(5) 
		(5)--(4) ; 
		
		\end{tikzpicture}
	\end{center}
	\caption{The binary tree $\psi(u)$ for $u = 31687524$}\label{fig_binary_tree_and_permutation}
\end{figure}
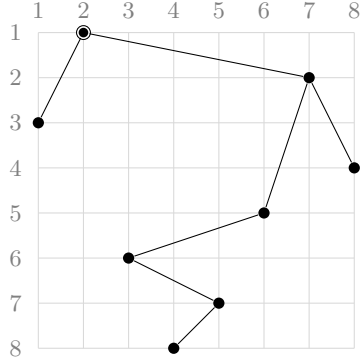

\begin{figure}[h]
\begin{tikzpicture}[scale = 0.6, dot/.style = {circle, fill=black, inner sep = 1.5pt, }]

\begin{scope}
\foreach\l[count=\y] in {3,...,1}{
 \draw[gray!30](1,\y)--(3,\y);
 \draw[gray!30] (\y,1)--(\y,3);
 \node[gray] at (0.5,\y) {\small$\l$};
 \node[gray] at (\y,3.5) {\small$\y$};
}

\node [black, draw, circle, inner sep = 1.5 pt, fill=black, double] (1) at (1,3) {};
\node[dot] (2) at (2,2) {};
\node[dot] (3) at (3,1) {}; 

\draw (1)--(2)--(3);

\node at (2,4.5) {$123$};
\end{scope}

\begin{scope}[xshift=3.5cm]
\foreach\l[count=\y] in {3,...,1}{
 \draw[gray!30](1,\y)--(3,\y);
 \draw[gray!30] (\y,1)--(\y,3);
 \node[gray] at (0.5,\y) {\small$\l$};
 \node[gray] at (\y,3.5) {\small$\y$};
}

\node [black, draw, circle, inner sep = 1.5 pt, fill=black, double] (1) at (1,3) {};
\node[dot] (2) at (2,1) {};
\node[dot] (3) at (3,2) {}; 

\draw (1)--(3)--(2);

\node at (2,4.5) {$132$};
\end{scope}

\begin{scope}[xshift=7cm]
\foreach\l[count=\y] in {3,...,1}{
 \draw[gray!30](1,\y)--(3,\y);
 \draw[gray!30] (\y,1)--(\y,3);
 \node[gray] at (0.5,\y) {\small$\l$};
 \node[gray] at (\y,3.5) {\small$\y$};
}

\node[dot] (1) at (1,2) {};
\node[dot] (2) at (2,1) {};
\node [black, draw, circle, inner sep = 1.5 pt, fill=black, double] (3) at (3,3) {}; 

\draw (3)--(1)--(2);

\node at (2,4.5) {$231$};
\end{scope}

\begin{scope}[xshift=10.5cm]
\foreach\l[count=\y] in {3,...,1}{
 \draw[gray!30](1,\y)--(3,\y);
 \draw[gray!30] (\y,1)--(\y,3);
 \node[gray] at (0.5,\y) {\small$\l$};
 \node[gray] at (\y,3.5) {\small$\y$};
}

\node[dot] (1) at (1,1) {};
\node[dot] (2) at (2,2) {};
\node [black, draw, circle, inner sep = 1.5 pt, fill=black, double] (3) at (3,3) {}; 

\draw (3)--(2)--(1);

\node at (2,4.5) {$321$};
\end{scope} 

\begin{scope}[xshift=14cm]
\foreach\l[count=\y] in {3,...,1}{
 \draw[gray!30](1,\y)--(3,\y);
 \draw[gray!30] (\y,1)--(\y,3);
 \node[gray] at (0.5,\y) {\small$\l$};
 \node[gray] at (\y,3.5) {\small$\y$};
}

\node[dot] (1) at (1,1) {};
\node [black, draw, circle, inner sep = 1.5 pt, fill=black, double] (2) at (2,3) {};
\node[dot] (3) at (3,2) {}; 

\draw (1)--(2)--(3);

\node at (2,4.5) {$312$};
\end{scope} 

\begin{scope}[xshift=17.5cm]
\foreach\l[count=\y] in {3,...,1}{
 \draw[gray!30](1,\y)--(3,\y);
 \draw[gray!30] (\y,1)--(\y,3);
 \node[gray] at (0.5,\y) {\small$\l$};
 \node[gray] at (\y,3.5) {\small$\y$};
}

\node[dot] (1) at (1,2) {};
\node [black, draw, circle, inner sep = 1.5 pt, fill=black, double] (2) at (2,3) {};
\node[dot] (3) at (3,1) {}; 

\draw (1)--(2)--(3);

\node at (2,4.5) {$213$};
\end{scope} 

\end{tikzpicture}
\caption{Permutations $u \in \mathfrak{S}_3$ and the corresponding binary trees $\psi(u)$.}\label{fig_permutation_and_tree}
\end{figure}
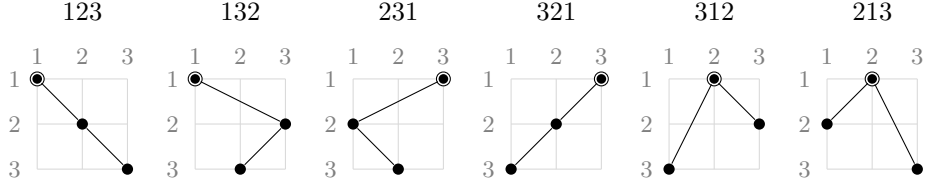

Because there is a canonical bijection between the set of binary trees with $n$ vertices and that of triangulations of $\polygon{n+2}$, we identify them so that we obtain the following. 

\begin{proposition}[{cf.~\cite[Appendix~A]{LFCG_dialgebras}}]\label{prop_permutations_and_binary_tress}\label{prop_psi}
	The assignment 
	\[
	\psi \colon \mathfrak{S}_n \twoheadrightarrow \{ \text{binary trees with $n$ vertices} \}=\{\text{triangulations of $\polygon{n+2}$}\} 
	\]
	is surjective. 
\end{proposition}

We denote by $s_i$ the \defi{simple transposition} in $\Sn$ interchanging $i$ and $i+1$ for $i =1,\dots,n$. We also denote by $t_{i,j}$ {$(i<j)$} the transposition interchanging two numbers $i$ and $j$, so $s_i = t_{i,i+1}$. For~$w \in \Sn$, an expression $w = s_{i_1} \cdots s_{i_{\ell}}$ is called \defi{reduced} if it is minimal among all such expressions. We denote by $\ell(w)$ the length of any reduced expression of $w$ and call it the \defi{length} of~$w$. The Bruhat order is a partial order on $\Sn$ defined to be $v \leq w$ if a reduced decomposition of $v$ is a substring of some reduced decomposition of $w$. We recall from~\cite[\S3.2]{BC12Singular} that for permutations~$v$ and~$w$ in $\mathfrak{S}_{n+1}$, we have
\begin{equation}\label{eq_Bruhat_order}
v \leq w \iff \{ v(1),\dots,v(i)\}\!\!\uparrow~ \leq \{ w(1),\dots,w(i)\}\!\!\uparrow \quad \text{ for all } i=1,\dots,{n}.
\end{equation}
Here, for a set $\{ a_1,\dots,a_i\}$ of distinct integers, $\{ a_1,\dots,a_i\}\!\!\uparrow$ denotes the ordered $i$-tuple obtained from the set $\{a_1,\dots,a_i\}$ by arranging its elements in ascending order. For two ordered tuples $(a_1,\dots,a_i)$ and $(b_1,\dots,b_i)$ of distinct numbers, we define $(a_1,\dots,a_i) \leq (b_1,\dots,b_i)$ if $a_k \leq b_k$ for all $1 \leq k \leq i$. 

For elements $v$ and $w$ in $\mathfrak{S}_{n+1}$ with $v \leq w$, the \defi{Bruhat interval polytope} 
$\Q_{v,w}$ is defined as follows:
\[
\Q_{v,w} = \operatorname{Conv}\{ (z(1),z(2),\dots,z(n+1)) \in \R^{n+1} 
\mid v \leq z \leq w\}.
\]
It is known that 
\begin{equation} \label{eq:Qvw_dimension}
\dim_{\R} \Q_{v,w}\le \ell(w)-\ell(v)
\end{equation}
and $\Q_{v,w}$ is called \textit{toric} if the equality above holds, i.e. $\dim_{\R} \Q_{v,w}=\ell(w)-\ell(v)$. (We will see in the next section that $\Q_{v,w}$ is toric if and only if the corresponding Richardson variety is a toric variety.) 
The Bruhat interval polytope $\Q_{v,w}$ is {combinatorially equivalent to} a cube if and only if it is toric and simple (\cite{TsukermanWilliams}, \cite{LMP2021}). There are many pairs $(v,w)$ such that $\Q_{v,w}$ is {combinatorially equivalent to} a cube. 
In the following, we will consider such pairs associated with a permutation in $\mathfrak{S}_n$. 

For a permutation $u \in \mathfrak{S}_n$, we define permutations $\uh{u}$ and $\ut{u}$ in $\Sn$ by
\[
\uh{u}(i) = \begin{cases}
1 & \text{ if } i = 1, \\
u(i-1)+1 & \text{ if }2 \leq i \leq n+1, 
\end{cases}
\qquad
\ut{u}(i) = \begin{cases}
u(i) & \text{ if } 1 \leq i \leq n, \\
n+1 & \text{ if } i =n+1.
\end{cases}
\]
For example, if $u = 2314$, then $\uh{u} = 13425$ and $\ut{u} = 23145$. As one may see, the permutation $\uh{u}$ is obtained from $u$ by putting the additional number $1$ at the \textit{head} while $\ut{u}$ is obtained from $u$ by putting the additional number $n+1$ at the \textit{tail}. In the notation, $\head$ stands for \textit{head} and $\tail$ stands for \textit{tail}. 

We set
\[
s(1,n) = s_1 s_2 \cdots s_{n}, \quad s(n,1) = s_{n} s_{n-1} \cdots s_1.
\]
One notes that if a pair $(v,w)$ of elements in $\Sn$ satisfies 
\[
w = v s(1,n)\quad \text{(resp. $w=vs(n,1)$)} \quad \text{ and } \quad\ell(w) = \ell(v) + n,
\]
then $v(1)=1$ (resp. $v(n+1)=n+1$), so that 
\begin{equation} \label{eq:pairvw}
\text{$(v,w) = (\uh{u}, \uh{u} s(1,n))$ (resp. $(v,w)=(\ut{u}, \ut{u} s(n,1)))$\quad for some $u \in \mathfrak{S}_{n}$.}
\end{equation}

It is shown in~\cite[\S 5 and \S 6]{HHMP19} (see also \cite{LMP2021}) that the Bruhat interval polytope $\Q_{v,w}$ for the pair~$(v,w)$ in \eqref{eq:pairvw} is {combinatorially equivalent to} an $n$-cube and our concern in this section is the pairs in \eqref{eq:pairvw}. We first consider the former case $(v,w) = (\uh{u}, \uh{u} s(1,n))$. 

\begin{proposition}\label{prop_u_head_atoms_and_lTrT}
	Let $u \in \mathfrak{S}_n$ and let $T = \psi(u)$ be the corresponding triangulation of $\polygon{n+2}$. We denote by $\lT$ and $\rT$ the left and right trees of $T$ as before. Then the edges of $\lT$ correspond to the atoms of the Bruhat interval $[\uh{u}, \uh{u}s(1,n)]$ while the edges of $\rT$ correspond to the coatoms of the Bruhat interval $[\uh{u}, \uh{u}s(1,n)]$. More precisely, 
	\[
	\begin{split}
	&\{{(i,j)} \mid \uh{u} \lessdot \uh{u} t_{i,j} \leq \uh{u} s(1,n) \} =
	\{ {(i,j)} \mid \{i-1,j-1\} \in E(\lT)\}, \\
	&\{{(i,j)} \mid \uh{u} \leq \uh{u} s(1,n) t_{i,j} \lessdot \uh{u} s(1,n) \} 
	= \{{(i,j)} \mid \{i,j\} \in E(\rT)\}.
	\end{split}
	\]
	Here, $x \lessdot y$ means that $x<y$ and there does not exist $z$ such that 
	$x < z < y$. 
\end{proposition}

\begin{proof}
	
	It follows from~\cite[\S 7 and \S8]{LMP2021} that the atoms and coatoms of the Bruhat interval \linebreak $[v,w]=[\uh{u}, \uh{u}s(1,n)]$ are given as follows: 
	\begin{enumerate}
		\item 
		For each $j=2,\dots, n+1$, there exists a unique $i$ such that
		\[
		i < j, \quad v(i) < v(j), \quad v(k) > v(j) \text{ for any }i < k < j. 
		\]
		The element $v t_{i,j}$ is an atom of the interval $[v,w]$. 
		The $n$ elements obtained in this way form the set of atoms of $[v,w]$.
		\item 
		For each $i=1,\dots,n$, there exists a unique $j$ such that
		\[
		i < j, \quad w(i) > w(j), \quad w(i) < w(k) \text{ for any }i < k < j. 
		\]
		The element $w t_{i,j}$ is a coatom of the interval $[v,w]$. 
		The $n$ elements obtained in this way form the set of coatoms of $[v,w]$.
	\end{enumerate}

We compare this with the description of the edges $E(\lT)$ and $E(\rT)$. Let $k = u^{-1}(1)$, that is, $v(k+1)=2$. Then, by the above description, $\{1,k+1\}$ provides an atom. Moreover, we have $w^{-1}(n+1) = 1$ and $w^{-1}(2) = k$. This implies that $\{k,n+1\}$ provides a coatom. On the other hand, consider the binary tree $\psi(u)$ and the corresponding polygon triangulation $T$. By definition of $\psi(u)$ and the association between rooted binary trees and polygon triangulations described in Section~\ref{sec_Catalan_numbers},  we have a triangle $\{0,k,n+1\}$ (having a distinguished edge $\{0,n+1\}$) in the $(n+2)$-gon triangulation $T$. Indeed, $\{0,k\} \in E(\lT)$ and $\{k,n+1\} \in E(\rT)$.  Continuing this process, i.e., choosing the minimum at each step, inductively for permutations given by $u(1) \cdots u(k-1)$ and $u(k+1) \cdots  u(n)$, we get the corresponding triangulation of the $(n+2)$-gon and this provides the set of edges of the left tree and that of the right tree. Hence the result follows. 
\end{proof}

\begin{example}\label{example_atom}
	Let $u = 31687524$. Then, the triangulation $T=\psi(u)$ is as shown in Figure~\ref{fig_atom_and_coatom_together}, $v=\uh{u}= 142798635$ and $w=\uh{u} s(1,8) = 427986351$. 
	There are $8$ atoms of the interval $[v,w]$ given by $v t_{i,j}$ where $(i,j)$ is one of the following pairs:
	\[
	(1,2), (1,3), (3,4), (4,5), (4,6), (3,7), (3,8), (8,9).
	\] 
	These pairs provide the edges of $\lT$ by subtracting $1$ from every component, which are
	\[
	(0,1), (0,2), (2,3), (3,4), (3,5), (2,6), (2,7), (7,8).
	\]
	On the other hand, there are $8$ coatoms given by $w t_{i,j}$ where $(i,j)$ is one of the following pairs: 
	\[
	(1,2), (2,9), (3,6), (4,5), (5,6), (6,7), (7,9), (8,9).
	\]
	These pairs are the edges of $\rT$.
\end{example}

\begin{theorem} \label{theo:normal_fan_s(1,n)}
	For $u\in \mathfrak{S}_n$, the normal fan of the Bruhat interval polytope $\Q_{\uh{u},\uh{u}s(1,n)}$ is the fan~$\Sigma_{\psi(u)}$ associated with the triangulation $\psi(u)$ of $\polygon{n+2}$. 
\end{theorem}

\begin{proof}
	As remarked before, $\Q_{\uh{u},\uh{u}s(1,n)}$ is {combinatorially equivalent to} an $n$-cube. The vertices of $\Q_{\uh{u},\uh{u}s(1,n)}$ correspond to the elements in the Bruhat interval $[\uh{u},\uh{u}s(1,n)]$ and the $n$ edges emanating from the vertex $\uh{u}$ (resp. $\uh{u}s(1,n)$) of $\Q_{\uh{u},\uh{u}s(1,n)}$ correspond to the atoms (resp. coatoms) of $[\uh{u},\uh{u}s(1,n)]$. Therefore, it follows from Proposition~\ref{prop_u_head_atoms_and_lTrT} and the definition of a Bruhat interval polytope that the primitive edge vectors emanating from the vertex $\uh{u}$ (resp. $\uh{u}s(1,n)$) are given by $\mathbf{p}_1,\dots,\mathbf{p}_n$ (resp. $\mathbf{q}_1,\dots,\mathbf{q}_n$) defined in Section~\ref{sec_left_and_right_trees}, where the triangulation $T$ of~$\polygon{n+2}$ is~$\psi(u)$. Their dual vectors are the vectors $\mathbf{v}_1,\dots,\mathbf{v}_n$ (resp. $\mathbf{w}_1,\dots,\mathbf{w}_n$) by Proposition~\ref{prop_dual}, so these vectors are ray generators of the normal fan of $\Q_{\uh{u},\uh{u}s(1,n)}$ {which is combinatorially equivalent to an $n$-cube}. To complete the proof of the theorem, we need to see which pairs of facets of $\Q_{\uh{u},\uh{u}s(1,n)}$ are in an opposite position. 
	
	The $n$ atoms and coatoms of the Bruhat interval $[\uh{u}, \uh{u} s(1,n)]$ are 
	\[
	\begin{split}
	\text{atoms: } & \uh{u} t_{\ast ,2}, \uh{u} t_{ \ast ,3},\dots, \uh{u} t_{ \ast ,n+1},\\
	\text{coatoms: } & \uh{u} s(1,n) t_{1, \ast}, \uh{u} s(1,n) t_{2,\ast} ,\dots, \uh{u} s(1,n) t_{n,\ast},
	\end{split}
	\]
	where $\ast$ means an appropriate number. We claim that
	\begin{equation}\label{eq_vt_and_wt}
	\uh{u} t_{\ast,j} \nleq \uh{u} s(1,n) t_{j-1,\ast} \quad \text{ for }j=2,\dots,n+1.
	\end{equation}
	We may assume that $\uh{u} t_{\ast,j} = \uh{u} t_{i,j}$ and $\uh{u} s(1,n) t_{j-1,\ast} = \uh{u} s(1,n) t_{j-1,k}$ for appropriate \linebreak $1 \leq i < j \leq n+1$ and $1 \leq j-1 < k \leq n+1$.
	We compare the first $j-1$ entries of $\uh{u} t_{i,j}$ and $\uh{u} s(1,n) t_{j-1,k}$:
	\[
	\arraycolsep = 1.4pt
	\begin{array}{rccccccccl}
	\uh{u} t_{i,j} : & \uh{u}(1),&\uh{u}(2),&\dots,&\uh{u}(i-1),&\uh{u}(j),&\uh{u}(i+1),&\dots,&\uh{u}(j-2)&\uh{u}(j-1),\ldots\\
	\uh{u} s(1,n) t_{j-1,k}: & \uh{u}(2),&\uh{u}(3),&\dots,&\uh{u}(i),&\uh{u}(i+1),&\uh{u}(i+2),&\dots,&\uh{u}(j-1)&\uh{u}(k+1),\ldots
	\end{array}
	\]
	Here, we have that $\uh{u}(i) < \uh{u}(j)$ since $\uh{u} t_{i,j} > \uh{u}$. Also, we get
	\[
	\uh{u}(j) = \uh{u} s(1,n)(j-1) > \uh{u} s(1,n)(k)=\uh{u}(k+1)
	\] 
	since $\uh{u}s(1,n) t_{j-1,k} < \uh{u} s(1,n)$. Therefore, we have that
	\[
	\{\uh{u}(1), \uh{u}(j) \}\!\!\uparrow~ \nleq \{\uh{u}(i), \uh{u}(k+1)\}\!\!\uparrow.
	\]
	This proves the claim \eqref{eq_vt_and_wt} together with~\eqref{eq_Bruhat_order}.
	
	The claim~\eqref{eq_vt_and_wt} implies that the set of primitive collections of the normal fan of $\Q_{\uh{u}, \uh{u} s(1,n)}$ is $\{\{\mathbf{v}_k,\mathbf{w}_k\} \mid k \in [n]\}$. This together with Lemma~\ref{lemm:fan} proves the theorem. 
\end{proof}

As for the latter case $(v,w)=(\ut{u},\ut{u}s(n,1))$, we have the following. 

\begin{theorem} \label{theo:normal_fan_s(n,1)}
	For $u\in \mathfrak{S}_n$, the normal fan of the Bruhat interval polytope $\Q_{\ut{u},\ut{u}s(n,1)}$ is isomorphic to the fan $\Sigma_{\psi(w_0uw_0)}$ associated with the triangulation $\psi(w_0uw_0)$ of $\polygon{n+2}$, where $w_0$ denotes the longest element in $\mathfrak{S}_n$. 
\end{theorem}

\begin{proof}
	First, we note that $\Q_{v,w}$ and $\Q_{w_0vw_0,w_0ww_0}$ are isomorphic {as lattice polytopes} for any $v,w\in\Sn$ with $v\le w$, where the same notation $w_0$ as above is used for the longest element in~$\Sn$. Indeed, the isomorphism is given by the linear automorphism of $\R^{n+1}$ defined by 
	\[
	(x_1,\dots,x_{n+1})\to (n+2-x_{n+1},\dots,n+2-x_1)
	\]
	because $(w_0zw_0)(i)=n+2-z(n+2-i)$ for any $z\in\Sn$ and $i=1,\dots,n+1$. Therefore, the normal fan of $\Q_{v,w}$ is isomorphic to that of $\Q_{w_0vw_0,w_0ww_0}$. 
	
	On the other hand, since $w_0\ut{u}w_0=\uh{(w_0uw_0)}$ as is easily checked, we have 
	\[
	(w_0\ut{u}w_0,w_0\ut{u}s(n,1)w_0)=(\uh{(w_0uw_0)},\uh{(w_0uw_0)}s(1,n)).
	\]
	This together with Theorem~\ref{theo:normal_fan_s(1,n)} implies the theorem. 
\end{proof}

\begin{remark}
	For $u\in\mathfrak{S}_n$, binary trees $\psi(u)$ and $\psi(uw_0)$ are isomorphic as unordered rooted trees. This is because since $(uw_0)(i)=u(n+1-i)$ for $i=1.\dots,n$, $\psi(uw_0)$ is the binary tree obtained by reflecting $\psi(u)$ about the vertical line passing through the root. This operation corresponds to the reflection of~$\polygon{n+2}$ through the line which cuts the edge joining the vertices $0$ and $n+1$ vertically. For example, this vertical line for $\polygon{10}$ and the reflection are given in Figure~\ref{fig_involution_on_T_and_B}. Therefore, the fan~$\Sigma_{\psi(w_0uw_0)}$ in Theorem~\ref{theo:normal_fan_s(n,1)} is isomorphic to the fan $\Sigma_{\psi(w_0u)}$ associated with $\psi(w_0u)$. However, $\psi(u)$ and $\psi(w_0u)$ are not isomorphic as unordered rooted trees in general, so the fans $\Sigma_{\psi(u)}$ and $\Sigma_{\psi(w_0u{w_0})}$ are not isomorphic in general. 
\end{remark}
\begin{figure}[hbtp]
	\begin{subfigure}{0.45\textwidth}
		\begin{center}
			\begin{tikzpicture}[scale=0.5]
			\foreach \x in {2,...,11}{
				\coordinate (\x) at (3*36-\x*36:3cm) ; 
			} 
			
			
			
			\draw[thick, dashed] (10)--(11) ; 
			\draw (11)--(8) ; 
			\draw
			(11) edge (2)
			(11) edge (3)
			(3) edge (4)
			(5) edge (6)
			(8) edge (9)
			(4) edge (5)
			(4) edge (7);
			\draw (10) -- (8) ; 
			\draw
			(10) edge (9)
			(8) edge (7)
			(8) edge (4)
			(8) edge (3)
			(7) edge (6)
			(7) edge (5)
			(3) edge (2) ;
			
			\draw[thick, red] (90:4cm) -- (-90:4cm);
			
			\end{tikzpicture} \hspace{0.5cm}
			\begin{tikzpicture}[every node/.style = {circle, fill=black, inner sep = 1.5pt, },
			level distance=5mm, 
			level 1/.style={sibling distance=13mm}, 
			level 2/.style={sibling distance=7mm},
			level 3/.style={sibling distance=5mm},
			scale=1.2 
			]
			\node[black, draw, circle, inner sep = 1.5 pt, fill=black, double]{}
			child{
				node{}
			}
			child{
				node{}
				child{
					node{}
					child{
						node{}
						child[missing] {}
						child{
							node{}
							child{
								node{}
							}
							child[missing] {}
						}
					}
					child[missing] {}
				}
				child{
					node{}
				}
			};
			
			\end{tikzpicture}
		\end{center}
		\caption{$T$ and $\B$}\label{fig_T_and_B_example_1}
	\end{subfigure}
	\begin{subfigure}{0.45\textwidth}
		\begin{center}
			\begin{tikzpicture}[scale=0.5]
			\foreach \x in {2,...,11}{
				\coordinate (\x) at (-8*36+\x*36:3cm) ; 
			} 
			\draw (11)--(8) ;
			\draw
			(11) edge (2)
			(11) edge (3)
			(3) edge (4)
			(5) edge (6)
			(8) edge (9)
			(4) edge (5)
			(4) edge (7);

			\draw[thick, dashed] (10)--(11) ; 
			\draw (10) -- (8) ;
			\draw
			(10) edge (9)
			(8) edge (7)
			(8) edge (4)
			(8) edge (3)
			(7) edge (6)
			(7) edge (5)
			(3) edge (2) ;
			\draw[thick, red] (90:4cm) -- (-90:4cm);
			\end{tikzpicture} \hspace{0.5cm}
			\begin{tikzpicture}[every node/.style = {circle, fill=black, inner sep = 1.5pt, },
			level distance=5mm, 
			level 1/.style={sibling distance=13mm}, 
			level 2/.style={sibling distance=7mm},
			level 3/.style={sibling distance=5mm},
			scale=1.2 
			]
			\node[black, draw, circle, inner sep = 1.5 pt, fill=black, double]{}
			child{
				node{}
				child{
					node{}
				}
				child{
					node{}
					child[missing] {}
					child{
						node{} 
						child{
							node{}
							child[missing] {}
							child{
								node{}
							}
						}
						child[missing] {}
					}
				}
			}
			child{
				node{}
			}; 
			
			\end{tikzpicture}
		\end{center}
		\caption{{Reflection image of $T$ and $\B$}}\label{fig_T_and_B_example_2}
	\end{subfigure}
	\caption{{Reflection of} a triangulated polygon and a binary tree.}\label{fig_involution_on_T_and_B}
\end{figure}
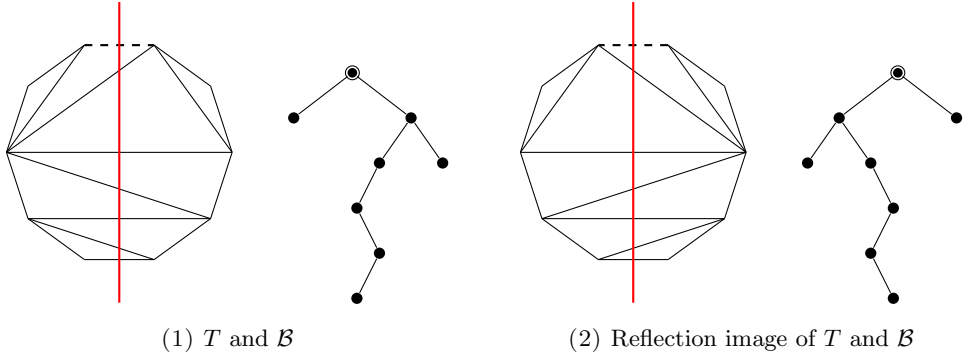

\begin{remark}
For each triangulation $T$ of $\polygon{n+2}$, there exists a Bruhat interval polytope whose normal fan is the fan $\Sigma_{T}$ by the surjectivity of $\psi$ and 
Theorems~\ref{theo:normal_fan_s(1,n)} and~\ref{theo:normal_fan_s(n,1)}.
On the other hand, for $n \leq 4$, whenever the normal fan of a Bruhat interval polytope $\Q_{v,w}$ is of Catalan type, the pair $(v,w)$ is of the form $(\uh{u}, \uh{u} s(1,n))$ or $(\ut{u}, \ut{u} s(n,1))$. 
We do not know whether this is true for any $n$.
\end{remark}

\begin{remark}
	One may wonder whether the normal fan of a Bruhat interval polytope $\Q_{v,w}$ is of Catalan type whenever $\Q_{v,w}$ is {combinatorially equivalent to} a cube. However, this is not true. For example, when $v$ is the identity element and $w=s_3s_1s_2=2413$, the Bruhat interval polytope~$\Q_{v,w}$ is {combinatorially equivalent to} a $3$-cube and its normal fan has ray generators 
	\[
	\begin{array}{lll}
	\mathbf{v}_1 = \varpi_1, &\quad \mathbf{v}_2 = \varpi_2, &\quad \mathbf{v}_3 = \varpi_3, \\
	\mathbf{w}_1 = -\varpi_1, &\quad \mathbf{w}_2 = \varpi_1-\varpi_2+\varpi_3,&\quad \mathbf{w}_3 = -\varpi_3.
	\end{array}
	\]
	Since $\mathbf{v}_2 + \mathbf{w}_2 = \mathbf{v}_1 +\mathbf{v}_3$, the primitive relations do not satisfy even the Fano condition, so it is not of Catalan type. 
\end{remark}


\section{Smooth toric Richardson varieties of Catalan 
type}\label{sec_smooth_toric_Richadson}

In this section, we interpret the results obtained in the previous sections in terms of Richardson varieties.
We first recall the definition of Richardson varieties and their properties. Let $G = \GL_{n+1}(\C)$, $B \subset G$ the {subgroup consisting} of upper triangular matrices, and $\Tb \subset B$ the {subgroup consisting} of diagonal matrices. The homogeneous space $G/B$ is called the \defi{flag variety} which can be identified with
\[
\flag_{n+1} \coloneqq \{ (\{0\} \subsetneq V_1 \subsetneq V_2 \subsetneq \cdots \subsetneq V_{n+1} = \C^{n+1}) \mid \dim_{\C} V_i = i \quad \text{ for }i = 1,\dots,n+1 \}.
\]
The left multiplication by $\Tb$ on $G$ induces the $\Tb$-action on $G/B$. 
The set of $\Tb$-fixed points in $G/B$ bijectively corresponds to the symmetric group $\Sn$ on $[n+1]$.
Indeed, for each $w \in \Sn$, we have 
\[
e_w \coloneqq 
(\{0\} \subsetneq \langle e_{w(1)}\rangle \subsetneq \langle e_{w(1)},e_{w(2)}\rangle\subsetneq \cdots\subsetneq \C^{n+1}) \in (\flag_{n+1})^\Tb
\]
which is called a \defi{coordinate flag}. Here, $e_1,\dots,e_{n+1}$ are the standard basis vectors of~$\C^{n+1}$. There is a moment map 
\begin{equation} \label{eq:moment_map}
\mu\colon G/B\to \R^{n+1}
\end{equation}
which sends the coordinate flag $e_w$ to the point $(w^{-1}(1),\dots,w^{-1}(n+1))$ in $\R^{n+1}$ (see \cite[\S 3]{LMP2021}). 

For an element $w \in \Sn$, we denote by~$X_w$ the \defi{Schubert variety} $\overline{BwB/B}$ in the flag variety~$G/B$. For a pair $(v,w)$ of elements in $\Sn$ with $v \leq w$, the \defi{Richardson variety}
$\Xvw{v}{w}$ is defined by the intersection $X_w \cap w_0X_{w_0 v}$. 
For example, $\Xvw{e}{w} = X_w$ and $\Xvw{e}{w_0} = X_{w_0} = G/B$, where $e$ denotes the identity element in $\Sn$.
Each Richardson variety is a $\Tb$-invariant irreducible variety and
\begin{equation} \label{eq:Xvw_dimension}
\dim_{\C} \Xvw{v}{w} = \ell(w) - \ell(v).
\end{equation}

Through the correspondence $e_w\to w$, the set of $\Tb$-fixed points in $\Xvw{v}{w}$ corresponds to the Bruhat interval~$[v,w] =  \{ z \in \Sn \mid v \leq z \leq w \}$ (cf.~\cite[\S 10.5]{Fulton97Young}).
Therefore, by the convexity theorem due to Atiyah~\cite{Atiyah82} (or Guillemin--Sternberg~\cite{GuilleminSternberg82}), we have 
\begin{equation}\label{equation_moment_map_image_of_Xvw}
\mu(\Xvw{v}{w})= \text{Conv}\{\mu(z) \mid v \le z \le w\} = \Q_{v^{-1},w^{-1}}
\end{equation}
because $\mu(e_z)=(z^{-1}(1),\dots,z^{-1}(n+1))$ for any $z\in\Sn$. 
Recall from~\cite[Proposition~7.12]{TsukermanWilliams} that the $X^v_w$ is a toric variety with respect to the $\Tb$-action if and only if 
\begin{equation*} \label{eq:Qvw_dimension2}
\dim_{\R} \Q_{v^{-1},w^{-1}} = \ell(w^{-1})-\ell(v^{-1})=\ell(w)-\ell(v).
\end{equation*}
Therefore, when $\Xvw{v}{w}$ is toric, its fan is the normal fan of $\Q_{v^{-1},w^{-1}}$. 

Every toric Schubert variety is smooth but not all toric Richardson varieties are smooth (see~\cite{LMP2021}). Smooth toric Richardson varieties are characterized in term of their moment polytopes as follows. 

\begin{proposition}[{\cite[Proposition~1.2, Theorem~1.3, Corollary~5.8]{LMP2021}}] \label{prop_review_on_toric_BIP}
	A Richardson variety $X_w^v$ is toric and smooth if and only if $\Q_{v^{-1},w^{-1}}$ is combinatorially equivalent to a cube \textup{(}equivalently, $\Q_{v,w}$ is combinatorially equivalent to a cube\textup{)}. 
\end{proposition}

We say that a smooth toric Richardson variety $\Xvw{v}{w}$ is \emph{of Catalan type} if it is of Catalan type as a toric variety, in other words, if the normal fan of $\Q_{v^{-1},w^{-1}}$ is of Catalan type. Since any \linebreak $n$-dimensional fan of Catalan type is realized as the normal fan of $\Q_{v^{-1},w^{-1}}$ with \linebreak $(v^{-1},w^{-1})=(\uh{u},\uh{u}s(1,n))$ for some $u\in \mathfrak{S}_n$ by Theorem~\ref{theo:normal_fan_s(1,n)}, we obtain the following as a direct consequence of Corollary~\ref{cor_enumerate}. 

\begin{corollary}\label{cor_enumeration_Xvw}
	The number of isomorphism classes of $n$-dimensional smooth toric Richardson varieties of Catalan type is the Wedderburn--Etherington number $b_{n+1}$. 
\end{corollary}

We shall consider a wider family of smooth toric Richardson varieties. Since $\mu(\Xvw{v}{w})=\Q_{v^{-1},w^{-1}}$, we shall use $\Xvw{v^{-1}}{w^{-1}}$ instead of $\Xvw{v}{w}$ so that $\mu(\Xvw{v^{-1}}{w^{-1}})=\Q_{v,w}$ and we can apply the results in the previous sections or in  \cite{LMP2021} directly. 

It is known that $\Q_{e,w}$ is toric if and only if $w$ is a product of distinct simple transpositions.  Thus, it is natural to study a pair $(v,w)$ such that 
\begin{equation}\label{eq_v_w_special_form}
v^{-1}w=s_{j_1} s_{j_2} \cdots s_{j_m} \quad \text{ with } \ell(w) - \ell(v) = m,
\end{equation}
where $j_1,\dots,j_m$ are {mutually} distinct. Note that the pair $(v,w)$ treated in the previous section is a special case where $v^{-1}w=s(1,n)$ or $s(n,1)$. It is shown in \cite[Proposition 7.1]{LMP2021} that $\Q_{v,w}$ for the pair $(v,w)$ in \eqref{eq_v_w_special_form} is toric. However, such $\Q_{v,w}$ is not necessarily combinatorially equivalent to a cube (see \cite[\S 7]{LMP2021}) although it is combinatorially equivalent to a cube when $v=e$. We recall from~\cite[\S 7]{LMP2021} a sufficient condition such that $\Q_{v,w}$ is combinatorially equivalent to a cube. To state it, we set up some notation and terminology. 

For $p$ and $q$ in $[n]$, we set
\[
s(p,q) = \begin{cases}
s_p s_{p+1} \cdots s_q & \text{ when } p \leq q, \\
s_p s_{p-1} \cdots s_q &\text{ when } p \geq q.
\end{cases}
\]
For each $s(p,q)$, we set 
\[
\bar{p}= \min\{p,q\},\quad \bar{q} = \max\{p,q\}.
\]
We note that if $j_1,\dots,j_m \in [n]$ are mutually distinct, then we have a \defi{minimal} expression
\begin{equation}\label{eq_minimal_expression}
s_{j_1} s_{j_2} \cdots s_{j_m} = s(p_1,q_1) s(p_2,q_2) \cdots s(p_r,q_r),
\end{equation}
where the intervals $[\bar{p}_1,\bar{q}_1],\dots,[\bar{p}_r,\bar{q}_r]$ are disjoint and $r$ is the minimum among such expressions. 
We say that the product $s_{j_1}s_{j_2}\cdots s_{j_m}$ in~\eqref{eq_minimal_expression} is \defi{proper} if no two intervals among $[\bar{p}_1,\bar{q}_1],\dots,[\bar{p}_r,\bar{q}_r]$ are adjacent. 
\begin{example} 
	\begin{enumerate}
		\item Suppose that $w = s_1s_2s_1 s_4s_5s_6$ and $v = s_1$. Then $v^{-1}w =s(2,1)s(4,6)$ and it is proper because $[1,2]$ and $[4,6]$ are not adjacent.
		\item Suppose that $w = s_1 s_2s_1 s_3s_4$ and $v = s_1$. Then $v^{-1}w =  s(2,1) s(3,4)$ and it is not proper because $[1,2]$ and $[3,4]$ are adjacent.
	\end{enumerate}
\end{example}

For $1 \leq p < q \leq n$, we define
\[
\pi_{[p,q]} \colon \Sn \to \mathfrak{S}_{q-p+2}
\]
by sending $w = w(1) w(2) \cdots w(n+1) \in \Sn$ to $u = u(1)u(2) \cdots u(q-p+2)$ such that the subsequence $w(p) w(p+1) \cdots w(q+1)$ has the same pattern as $u$. 
Here, we say that two sequences of numbers $a(1) a(2) \cdots a(k)$ and $b(1) b(2) \cdots b(k)$ \textit{have the same pattern} if for all $1 \leq i < j \leq k$, $a(i) < a(j)$ if and only if $b(i) < b(j)$.

\begin{example}\label{example_projection}
	Let $v = 173254689$ and $w = v s(1) s(4,3) s(6,8) = 715326894$. 
	We get $(\bar{p}_1,\bar{q}_1) = (1,1)$, $(\bar{p}_2,\bar{q}_2) = (3,4)$, and $(\bar{p}_3,\bar{q}_3) = (6,8)$. Since $v(1) v(2) = 17$, $v(3) v(4) v(5) = 325$, and $v(6) v(7) v(8) v(9) = 4689$, we have 
	\[
	\pi_{[1,1]}(v) = 12, \quad \pi_{[3,4]}(v) = 213, \quad \pi_{[6,8]}(v) =1234.
	\]
	Similarly, we get 
	\[
	\pi_{[1,1]}(w) = 21, \quad \pi_{[3,4]}(w) = 321,\quad \pi_{[6,8]}(w) = 2341.
	\] 
	Then one can see that 
	\[
	\Q_{v,w} = \Q_{12,21} \times \Q_{213, 321} \times \Q_{1234, 2341},
	\]
	where each factor is of Catalan type (meaning that its normal fan is of Catalan type). 
\end{example}

The observation above holds in general. Indeed, we have the following. 
\begin{proposition}[{\cite[Lemma~6.1 and Proposition~7.3]{LMP2021}}]\label{prop_Qvw_is_product} 
	Suppose that the minimal expression of $s_{j_1} s_{j_2} \cdots s_{j_m}$ in~\eqref{eq_minimal_expression} is proper. Then, for any pair $(v,w)$ such that $v^{-1}w =s_{j_1} s_{j_2} \cdots s_{j_m}$ and $\ell(w) - \ell(v) = m$, the Bruhat interval polytope $\Q_{v,w}$ is combinatorially equivalent to an $m$-cube. More precisely, if we set $(v_i,w_i)=(\pi_{[\bar{p}_i,\bar{q}_i]}(v), \pi_{[\bar{p}_i,\bar{q}_i]}(w))$ for $i=1,\dots,r$, then we have 
	\[
	\Q_{v,w}=\prod_{i=1}^r \Q_{v_i,w_i}\quad\text{and hence}\quad \Xvw{v^{-1}}{w^{-1}} =\prod_{i=1}^r \Xvw{v_i^{-1}}{w_i^{-1}}, 
	\]
	where each factor is of Catalan type. 
\end{proposition}
\begin{proof}
Recall from~\cite[Proposition~7.3]{LMP2021} that if the minimal expression of $s_{j_1} s_{j_2} \cdots s_{j_m}$ in~\eqref{eq_minimal_expression} is proper, then any pair $(v,w)$ in the statement defines the Bruhat interval polytope $\Q_{v,w}$ which is combinatorially equivalent to an $m$-cube. This proves the former  statement.

Furthermore, the Bruhat interval polytope $\Q_{v,w}$ is toric by~\cite[Theorem~5.7]{LMP2021}, that is, $\dim \Q_{v,w} = \ell(w) - \ell(v)$. 
On the other hand,  by~\cite[Theorem~4.1]{TsukermanWilliams}, every face of a Bruhat interval polytope is itself a Bruhat interval polytope; moreover, a Bruhat interval polytope $\Q_{x,y}$ forms a face of a toric Bruhat interval polytope $\Q_{v,w}$ for any $[x,y] \subset [v,w]$ by~\cite[Theorem~5.1]{LMP2021}. Therefore, the face structure of the toric Bruhat interval~$\Q_{v,w}$ considered in the statement is decided by its Bruhat subintervals. 
Because the minimal expression is proper, we have 
\[
[x,y] \subset [v,w] \iff [\pi_{[\bar{p}_i,\bar{q}_i]}(x), \pi_{[\bar{p}_i,\bar{q}_i]}(y)] \subset [v_i,w_i] \quad \text{ for all }i=1,\dots,r.
\]
This proves the latter statement.
\end{proof}

\begin{example}\label{example_corr_signed_rooted_forests}
	Let $v$ and $w$ be a pair in Example~\ref{example_projection}. Then 
	\[
	\begin{split}
	(v_1,w_1)&=(12,21)=(\uh{(1)},\uh{(1)}s(1,1)),\\
	(v_2,w_2)&=(213,321)=(\ut{(21)},\ut{(21)}s(2,1)),\\
	(v_3,w_3)&=(1234,2341)=(\uh{(123)},\uh{(123)}s(1,3)).
	\end{split} 
	\] 
	Therefore, it follows from Proposition~\ref{prop_Qvw_is_product}, 
	Theorems~\ref{theo:normal_fan_s(1,n)} and~\ref{theo:normal_fan_s(n,1)} that 
	\[
	\Xvw{v^{-1}}{w^{-1}} \cong {X(\Sigma_{\psi(1)}) \times X(\Sigma_{\psi(21)}) \times X(\Sigma_{\psi(123)})}.
	\]
	The binary forest $\psi(1) \sqcup \psi(21) \sqcup \psi(123)$ is shown as follows. 
	
	\medskip
	
	 
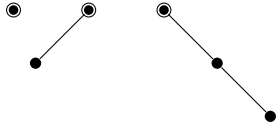
\begin{center}
	 \begin{tikzpicture}
	 \tikzstyle{vnode} = [circle, fill, inner sep = 1.5pt]
	 \node[black, draw, circle, inner sep = 1.5 pt, fill=black, double] (1) {};
	 \node[black, draw, circle, inner sep = 1.5 pt, fill=black, double, right of= 1] (3) {};
	 \node[vnode, below left of = 3] (4) {};
	 \node[black, draw, circle, inner sep = 1.5 pt, fill=black, double, right of = 3] (6) {};
	 \node[vnode, below right of = 6] (7) {};
	 \node[vnode, below right of = 7] (8) {};
	 
	 \draw (3)--(4);
	 \draw (6)--(7) ;
	 \draw (7)--(8);
	 \end{tikzpicture}
	 \end{center}
	 
\end{example}

Using the sequence $\{b_{n}\}$ of the Wedderburn--Etherington numbers, we can also enumerate the number of isomorphism classes of Richardson varieties $X^{v^{-1}}_{w^{-1}}$ such that $v^{-1}w$ has a proper minimal expression as follows. 

\begin{theorem}
	Let $f_m$ be the number of isomorphism classes of $m$-dimensional smooth toric Richardson varieties $\Xvw{v^{-1}}{w^{-1}}$ with $v, w \in \Sn$ such that $v^{-1}w$ has a proper minimal expression of distinct $m$ simple transpositions, where we assume $n\gg m$. Then 
	\begin{equation}\label{eq_fm}
	\sum_{m \geq 0} f_mx^m = \frac{1}{\prod_{k > 0} (1-x^k)^{b_{k+1}}}.
	\end{equation}
\end{theorem}
\begin{proof}
{Since} $b_{k+1}$ is {equal to} the number of unordered binary trees having $k$ vertices, the right-hand side of~\eqref{eq_fm} is the generating function of the number of unordered binary forests having $m$ vertices. So, it is enough to show that there is a bijective correspondence between the set of isomorphism classes of $m$-dimensional smooth toric Richardson varieties in the statement and the set of unordered binary forests having $m$ vertices. 
	
	The correspondence, denoted by $\Psi$, from the isomorphism classes of 
	$k$-dimensional smooth toric Richardson varieties of Catalan type to the set 
	of unordered binary tree having $k$ vertices is bijective by 
	Corollary~\ref{cor_enumeration_Xvw}. On the other hand, each smooth toric 
	Richardson variety in the statement is a product of toric Richardson varieties of 
	Catalan type by Proposition~\ref{prop_Qvw_is_product}. Therefore, the 
	bijective correspondence $\Psi$ induces a correspondence $\bar{\Psi}$ from 
	the set of isomorphism classes of $m$-dimensional Richardson varieties in the 
	statement to the set of unordered binary forests having $m$ vertices. The 
	correspondence $\bar{\Psi}$ is clearly injective. Using the surjectivity of $\Psi$ 
	and a natural embedding 
	\[
	\mathfrak{S}_{k_1}\times \cdots\times \mathfrak{S}_{k_r}\to \mathfrak{S}_{n+1},
	\]
	where $k_1+\dots+k_r=n+1$, one can see that $\bar{\Psi}$ is surjective if we pick a sufficiently large $n$. 
\end{proof}

The sequence $\{f_m\}$ is called \textit{Piet Hut's ``coat-hanger'' sequence}, which counts unlabeled rooted forests with $m$ edges such that each vertex has at most two children and the degree of each root is one
(see \href{http://oeis.org/A088325}{Sequence A088325} in OEIS \cite{OEIS}). We present $f_m$ for small values of $m$ in Table~\ref{table_fn}. Also, we draw all {unordered} binary forests having $m$ vertices for $m = 3,4$ in Figure~\ref{fig_unlabeled_forests_of_rooted_trees}.
\begin{table}[H]
	\begin{tabular}{c|rrrrrrrrrrrrrr}
		\toprule
		$m$ & $1$ & $2$ & $3$ & $4$ & $5$ & $6$ & $7$ & $8$ & $9$ & $10$ & $11$ & $12$ & $13$ & $14$ \\
		\midrule
		$f_m$ & $1$ & $2$ & $4$ & $8$ & $16$ & $34$ & $71$ & $153$ & $332$ & $730$ & $1617$ & $3620$ & $ 8148$ & $18473$ \\
		\bottomrule
	\end{tabular}
	\caption{Piet Hut's ``coat-hanger'' sequence $f_m$ for small values of $m$}\label{table_fn}
\end{table}

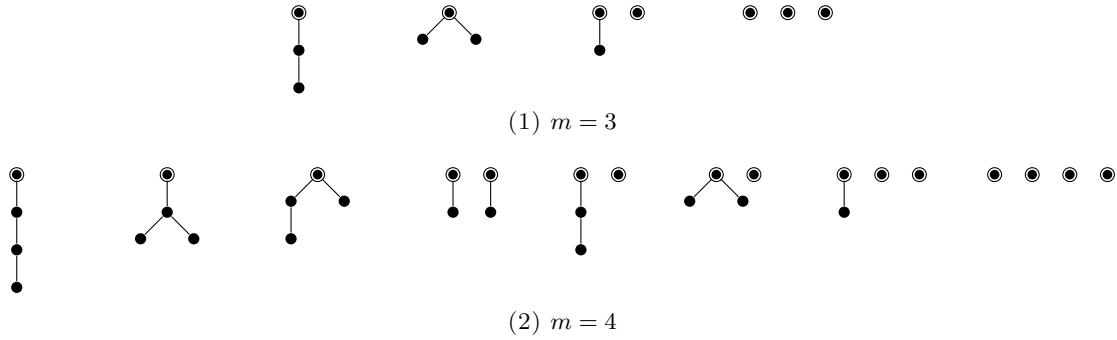
\begin{figure}[h]
 \begin{subfigure}{\textwidth}
 \centering
 \begin{tikzpicture}[ node distance = 0.5cm]
 \tikzstyle{vnode} = [circle, fill, inner sep = 1.5pt]
 \begin{scope}
 \node[black, draw, circle, inner sep = 1.5 pt, fill=black, double] (1) {};
 \node[circle, fill, inner sep = 1.5pt, below of =1] (2) {};
 \node[circle, fill, inner sep = 1.5pt, below of =2] (3) {};
 
 \draw (1)--(2); 
 \draw (2)--(3); 
 \end{scope}
 
 \begin{scope}[xshift=2cm]
 \node[black, draw, circle, inner sep = 1.5 pt, fill=black, double] (1) {};
 \node[circle, fill, inner sep = 1.5pt, below left of =1] (2) {};
 \node[circle, fill, inner sep = 1.5pt, below right of =1] (3) {};
 
 \draw (1)--(2) ;
 \draw (1)--(3) ;
 \end{scope}
 
 \begin{scope}[xshift=4cm]
 \node[black, draw, circle, inner sep = 1.5 pt, fill=black, double] (1) {};
 \node[circle, fill, inner sep = 1.5pt, below of =1] (2) {};
 \node[black, draw, circle, inner sep = 1.5 pt, fill=black, double, right of =1] (3) {};
 
 \draw (1)--(2);
 \end{scope}
 
 \begin{scope}[xshift=6cm]
 \node[black, draw, circle, inner sep = 1.5 pt, fill=black, double] (1) {};
 \node[black, draw, circle, inner sep = 1.5 pt, fill=black, double, right of =1] (2) {};
 \node[black, draw, circle, inner sep = 1.5 pt, fill=black, double, right of = 2] (3) {};
 \end{scope}
 \end{tikzpicture}
 \caption{$m=3$}
 \end{subfigure} 
 
 \vspace{1em}
 
 \begin{subfigure}{\textwidth}
 \centering
 \begin{tikzpicture}[ node distance = 0.5cm]
 \tikzstyle{vnode} = [circle, fill, inner sep = 1.5pt]
 \begin{scope}
 \node[black, draw, circle, inner sep = 1.5 pt, fill=black, double] (1) {};
 \node[vnode, below of= 1] (2) {};
 \node[vnode,below of= 2] (3) {};
 \node[vnode,below of=3] (4) {};
 
 \draw (1)--(2);
 \draw (2)--(3);
 \draw (3)--(4);
 \end{scope}
 
 \begin{scope}[xshift = 2cm]
 \node[black, draw, circle, inner sep = 1.5 pt, fill=black, double](1) {};
 \node[vnode, below of= 1] (2) {};
 \node[vnode, below left of =2] (3) {};
 \node[vnode, below right of = 2] (4) {};
 
 \draw (1)--(2) ;
 \draw (2)--(3);
 \draw (2)--(4) ;
 
 \end{scope}
 
 \begin{scope}[xshift=4cm]
 \node[black, draw, circle, inner sep = 1.5 pt, fill=black, double] (1) {};
 \node[vnode, below left of = 1] (2) {};
 \node[vnode, below right of = 1] (3) {};
 \node[vnode, below of= 2] (4) {}; 
 
 \draw (1)--(2) ;
 \draw (1)--(3) ;
 \draw (2)--(4) ;
 \end{scope}
 
 \begin{scope}[xshift=5.8cm]
 \node[black, draw, circle, inner sep = 1.5 pt, fill=black, double] (1) {};
 \node[vnode, below of=1] (2) {};
 \node[black, draw, circle, inner sep = 1.5 pt, fill=black, double, right of =1](3) {};
 \node[vnode, below of =3] (4) {};
 
 \draw (1)--(2) ;
 \draw (3)--(4) ;
 
 \end{scope}
 
 \begin{scope}[xshift=7.5cm]
 \node[black, draw, circle, inner sep = 1.5 pt, fill=black, double] (1) {};
 \node[vnode, below of= 1] (2) {};
 \node[vnode, below of= 2] (3) {};
 \node[black, draw, circle, inner sep = 1.5 pt, fill=black, double, right of= 1] (4) {};
 
 \draw (1)--(2) ;
 \draw (2)--(3) ;
 \end{scope}
 
 \begin{scope}[xshift = 9.3cm]
 \node[black, draw, circle, inner sep = 1.5 pt, fill=black, double] (1) {};
 \node[circle, fill, inner sep = 1.5pt, below left of =1] (2) {};
 \node[circle, fill, inner sep = 1.5pt, below right of =1] (3) {};
 \node[black, draw, circle, inner sep = 1.5 pt, fill=black, double, right of=1] (4) {};
 
 \draw (1)--(2) ;
 \draw (1)--(3) ;
 
 \end{scope}
 
 \begin{scope}[xshift=11cm]
 \node[black, draw, circle, inner sep = 1.5 pt, fill=black, double] (1) {};
 \node[circle, fill, inner sep = 1.5pt, below of =1] (2) {};
 \node[black, draw, circle, inner sep = 1.5 pt, fill=black, double, right of =1] (3) {};
 \node[black, draw, circle, inner sep = 1.5 pt, fill=black, double, right of =3] (4) {};
 \draw (1)--(2) ;
 \end{scope}
 
 \begin{scope}[xshift=13cm]
 \node[black, draw, circle, inner sep = 1.5 pt, fill=black, double] (1) {};
 \node[black, draw, circle, inner sep = 1.5 pt, fill=black, double, right of =1] (2) {};
 \node[black, draw, circle, inner sep = 1.5 pt, fill=black, double, right of =2] (3) {};
 \node[black, draw, circle, inner sep = 1.5 pt, fill=black, double, right of =3] (4) {};
 \end{scope}
 \end{tikzpicture}

 \caption{$m=4$}
 \end{subfigure}
 \caption{Unordered binary forests having $m$ vertices for $m = 3$ and $4$. We decorate the roots with double circles.}\label{fig_unlabeled_forests_of_rooted_trees}
\end{figure}

\begin{example}
	For each binary forest $\B = \B_1 \sqcup \cdots \sqcup \B_r$ having $r$ connected components in Figure~\ref{fig_unlabeled_forests_of_rooted_trees}, we provide a pair $(v,w)$ such that $\Xvw{v^{-1}}{w^{-1}} \cong {X(\Sigma_{\B_1}) \times \cdots \times X(\Sigma_{\B_r})}$.
	\begin{enumerate}
		\item For $m = 3$, from left to right, each pair $(v,w)$ is given by 
		\[
		(1234, 2341), \quad {(1324, 3241)}, \quad (12345, 23154), \quad (123456, 214365)
		\]
		\item For $m = 4$, from left to right, each pair $(v,w)$ is given by 
		\[
		\begin{split}
		&(12345, 23451), \quad (12435, 24351), \quad {(13524, 35241)}, 
		\quad (123456, 231564), \\
		&(123456, 234165), \quad {(132456,324165)}, 
		\quad (1234567, 2315476), \quad (12345678, 21436587).
		\end{split}
		\]
	\end{enumerate}
\end{example}

We close this section by addressing some questions. 
We proved in Lemma~\ref{lemm:Fano}  that every Richardson variety of Catalan type are smooth toric Fano varieties. 
Moreover, any smooth toric Fano Richardson variety $X^v_w$ for $v,w \in \mathfrak{S}_4$ is a product of toric varieties of Catalan type.\footnote{We have checked this using the computer program SageMath~\cite{sagemath}.} 
It is natural to ask the following. 
\begin{Question}\label{question}
Is any smooth toric Fano Richardson variety a product of toric varieties of Catalan type? 
\end{Question}

Flag varieties and Richardson varieties are defined for any Lie type so it would be interesting to extend the investigation of the connection between the combinatorics and toric Richardson varieties to other Lie types.
\begin{Question}\label{question_other_Lietype}
Does there exist a combinatorial object which describes and classifies toric Richardson varieties in other Lie types for the pairs $(v,w)$ in~\eqref{eq:simple_pair_vw}?
\end{Question}
\providecommand{\bysame}{\leavevmode\hbox to3em{\hrulefill}\thinspace}
\providecommand{\MR}{\relax\ifhmode\unskip\space\fi MR }
\providecommand{\MRhref}[2]{%
  \href{http://www.ams.org/mathscinet-getitem?mr=#1}{#2}
}
\providecommand{\href}[2]{#2}

\end{document}